\documentclass[12pt]{amsart}
\usepackage{amsfonts,amssymb,amscd,xy}
\usepackage[leqno]{amsmath}
\usepackage{mathrsfs}
\usepackage{amssymb}
\usepackage{amsmath}
\usepackage{amsxtra}
\usepackage{amsthm}
\DeclareMathAlphabet{\mathbbmsl}{U}{bbm}{m}{sl}

\usepackage[usernames,dvipsnames]{xcolor}
\usepackage{tabularx}
\usepackage{epsfig,amssymb}
\usepackage{bm} 

\usepackage{ulem}

\usepackage{hyperref} %needed for \url in bibliography
\definecolor{labelkey}{rgb}{1,0,0}

\makeatletter
\def\bm#1{\mathpalette\bmstyle{#1}}
\def\bmstyle#1#2{\mbox{\boldmath$#1#2$}}

\DeclareMathAlphabet{\mathcalligra}{T1}{calligra}{m}{n}

\newcommand{\sh}{\kern -.4em\phantom{a}^{\mathbf{\sim}}}

\newcommand{\lra}{\longrightarrow}

\newcommand{\fpqc}{{\rm fpqc}}
\newcommand{\fppf}{{\rm fppf}}

\newcommand{\red}{{\rm red}}

\newcommand{\kbar}{\overline{k}}
\def\be{\kern -.1em}
\def\bbe{\kern -.07em}
\def\le{\kern 0.03em}
\def\lle{\kern 0.015em}
\def\lbe{\kern -.025em}
\def\llbe{\kern -.03em}

\newcommand{\Z}{{\mathbb Z}}
\newcommand{\Q}{{\mathbb Q}}
\newcommand{\N}{{\mathbb N}}
\newcommand{\G}{{\mathbb G}}

\newcommand{\spec}{\mathrm{ Spec}\,}

\newcommand{\krn}{\mathrm{Ker}\e}
\newcommand{\img}{\mathrm{Im}\e}

\newcommand{\Hom}{{\mathrm{Hom}}}

\def\e{\kern 0.08em}

\newcommand{\s}{\mathscr }

\newcommand{\mr}{\mathrm }

\newcommand{\pf}{{\rm pf}}

\numberwithin{equation}{section}

\newtheorem{lemma}[equation]{Lemma} 
\newtheorem{theorem}[equation]{Theorem}
\newtheorem{proposition-definition}[equation]{Proposition-Definition}
\newtheorem{corollary}[equation]{Corollary}
\newtheorem{proposition}[equation]{Proposition}
\theoremstyle{definition}

\theoremstyle{remark}
\newtheorem{remark}[equation]{Remark}
\newtheorem{remarks}[equation]{Remarks}
\newtheorem{example}[equation]{Example}

\newtheorem{caveat}[equation]{Caveat}

\definecolor{labelkey}{rgb}{1,0,0}

\begin{document}

\input xy     
\xyoption{all}

\title{On the perfection of schemes}

\author{Alessandra Bertapelle}
\address{Universit\`a degli Studi di Padova, Dipartimento di Matematica, via Trieste 63, I-35121 Padova}
\email{alessandra.bertapelle@unipd.it}

\author{Cristian D. Gonz\'alez-Avil\'es}
\address{Departamento de Matem\'aticas, Universidad de La Serena, Cisternas 1200, La Serena 1700000, Chile}
\email{cgonzalez@userena.cl}
\thanks{G.-A. is partially supported by Fondecyt grant 1160004.}
\date{\today}

\begin{abstract}  We present a detailed and  elementary construction of the inverse perfection of a scheme and discuss some of its main properties. We also establish a number of auxiliary results (for example, on inverse limits of schemes) which do not seem to appear in the literature.
\end{abstract}

\maketitle

\topmargin -1cm

\section{Introduction}

Let $p$ be a prime number and let $\mathbb F_{\be p}$ denote the field with $p$ elements.  
In the course of our review of the construction of the perfect Greenberg functor in \cite{gfr}, we were hampered by the lack of an adequate reference work on the subject of (inverse) perfections of $\mathbb F_{\be p}$-schemes. Although the classical reference \cite{gre} presents in some detail the construction of the inverse perfection $Y^{\lbe\pf}$ of an $\mathbb F_{\be p}$-scheme $Y$, it does not discuss its main properties. On the other hand, the relatively recent preprint \cite{bs} briefly discusses some of the main properties of the indicated construction (see \cite[Lemmas 3.4 and 3.8]{bs}, parts of which overlap with some of the results presented here), but it does not address the perfection of $\mathbb F_{\be p}$-group schemes. Our aim in this paper is to present a detailed and elementary construction of the inverse perfection of an $\mathbb F_{\be p}$-scheme and discuss some of its properties. The (inverse) perfection functor has played, a continues to play, a significant role in algebraic geometry (see, for example, \cite{ser1, ser2,bd,bw, pep,kl}). We believe that our presentation will be useful to all students and researchers that at some point in their studies will need to consider the (inverse) perfection of an $\mathbb F_{\be p}$-scheme.

\smallskip

We briefly indicate the contents of the individual Sections.

Section \ref{fl} presents some basic results on the fpqc and fppf topologies. These statements may be well-known to some readers but, to our knowledge, they do not appear in the literature. Section \ref{plim-sec} discusses certain basic properties of projective limits of schemes that, to our surprise, we could not find in the standard literature on the subject. In particular, Proposition \ref{lim-surj} shows that, if $k$ is any field, then the inverse limit functor is exact on certain types of ``Mittag-Leffler" short exact sequences of projective systems {\it in the category of commutative $k$-group schemes}.
Section \ref{pclos} is a detailed discussion of the construction of the {\it perfect closure} (or {\it direct perfection}) of an $\mathbb F_{\be p}$-algebra. The developments of Section \ref{pclos} are then extended to the category of $\mathbb F_{\be p}$-schemes in Section \ref{perf}. The final Section \ref{exa} discusses exactness properties of the inverse perfection functor on the category of group schemes over a perfect field of positive characteristic $p$.

\smallskip

We will write $\vert X\vert$ for the underlying topological space of a scheme $X$. Further, if $S$ is a scheme, we will write $(\mathrm{Sch}/S\e)$ for the category of $S$-schemes.

\section{The flat topologies}\label{fl}
Let $S$ be a scheme and let $\mathcal C$ be a full subcategory of $(\mathrm{Sch}/S\e)$ which is stable under (fiber) products and contains the identity morphism of $S$. 
Recall from \cite[\S2.3.2, pp.~27-28]{v} that an {\it fpqc morphism} is a faithfully flat morphism of schemes $f\colon X\to Y$ with the following property:
if $x$ is a point of $X$, then there exists an open neighborhood $U$ of $x$
such that the image $f(U)$ is open in $Y$ and the induced morphism $U\to f(U)$ is quasi-compact. See \cite[Proposition 2.33, p.~27]{v} for a list of equivalent properties. Clearly, a faithfully flat and quasi-compact morphism of schemes is an $\fpqc$ morphism. Further, the class of fpqc morphisms is stable under base change \cite[Proposition 2.35(v), p.~28]{v}. An {\it fppf morphism} is a faithfully flat morphism locally of finite presentation. By \cite[Proposition 2.35(iv), p.~28]{v}, every fppf morphism is an fpqc morphism. The fpqc (respectively, fppf) topology on $\mathcal C$ is the topology where the coverings are collections of flat morphisms $\{ X_{\alpha}\to X\}$ in $\mathcal C$ such that the induced morphism $\coprod X_{\alpha}\to X$ is an fpqc (respectively, fppf) morphism. Clearly, the fpqc topology is finer than the fppf topology.
If $\tau=\fpqc$ or $\fppf$, we will write $\mathcal C_{\tau}$ for the corresponding site and $\mathcal C^{\,\sh}_{\tau}$    for the category of sheaves of sets  on $\mathcal C_{\tau}$.  
 For either of the sites mentioned above, every representable presheaf is a sheaf \cite[Theorem 2.55, p.~34]{v} and the induced functor  
\[
h_{\lbe
S}\colon \mathcal C\to \mathcal C^{\e\sh}_{\tau},Y\mapsto \Hom_{S}(-,Y),
\]
is fully faithful. A sequence $1\to F\to G\to H\to 1$ of groups in $\mathcal C$  will be called {\it exact for the $\tau$ topology on $\mathcal C$} if the corresponding sequence of sheaves of groups, namely $1\to h_{\lbe S}(F\le)\to h_{\lbe S}(G\lle)\to h_{\lbe S}(H\lle)\to 1$, is exact, i.e., the first nontrivial map identifies $h_{\lbe S}(F\le)$ with the kernel of $h_{\lbe S}(G)\to h_{\lbe S}(H)$ and the latter map is {\it locally surjective}, i.e.,   for every $U\in\mathcal C$ and $s\in h_{\lbe S}(H)(U)$, $s$ can be locally lifted to  $h_{\lbe S}(G)(U)$ in the $\tau$ topology \cite[p.~ 122, lines 22-24]{mi1}.

\begin{remarks} Let $q\colon G\to H$ be a morphism of groups in $\mathcal C$.
\begin{enumerate}
\item[(a)] It is not difficult to check that, if $h_{\lbe S}(q)\colon h_{\lbe S}(G\lle)\to h_{\lbe S}(H\lle)$ is locally surjective for the $\tau$ topology, then $h_{\lbe S}(q)$ is an epimorphism of sheaves for the indicated topology \cite[IV, Definition 1.1]{sga3}, i.e., the induced map of sets $\mr{Hom}(h_{\lbe S}\lbe(H\lle), \mathcal F\e)\to \mr{Hom}(h_{\lbe S}\lbe(G\lle), \mathcal F\e)$ is injective for every $\mathcal F\in\mathcal C^{\,\sh}_{\tau}$. For the converse, see \cite[IV, Remark 4.4.8]{sga3}.
\item[(b)] We will show below that, if $h_{\lbe S}(q)$ is locally surjective for the $\tau$ topology, i.e., an epimorphism of sheaves in the $\tau$ topology (see (a)), then $q$ is a surjective morphism of schemes. However, the converse may fail to hold. For example, let $p$ be a prime number, set $k=\mathbb F_{\be p}$ and let $q\colon 0\to \alpha_{p}$ be the canonical morphism, where $\alpha_{p}=\spec\be(k[x]/(x^{\le p}))$ is the kernel of the Frobenius endomorphism $\G_{a,\le k}\to \G_{a,\le k}$. Since $\alpha_{p}$ is a one-point scheme, $q$ is clearly a surjective morphism of schemes. However, it is easy to check that $h_{\lbe S}(q)$, where $S=\spec k$, is not locally surjective. 
\end{enumerate}
\end{remarks}

\begin{lemma}\label{surj} Let $\,1\to F\overset{\!i}{\to} G\overset{\!q}{\to} H\to 1$ be a sequence of groups in $\mathcal C$ which is exact for the $\tau$ topology on $\mathcal C$, where $\tau=\fppf$ or $\fpqc$. Then $i$ identifies $F$ with the scheme-theoretic kernel of $q$, i.e., $F=\krn q=G\times_{H}S$, and $q$ is a surjective morphism of schemes.
\end{lemma}
\begin{proof} Since $\krn q$ represents the functor $\krn\lbe(h_{\lbe S}(q)\colon h_{\lbe S}(G)\to h_{\lbe S}(H))=h_{S}(F\le)$, i.e., $h_{\lbe S}\lbe(\le\krn q)=h_{\lbe S}(F\le)$, the first assertion is clear. To show that $q\colon G\to H$ is  a surjective morphism of schemes, it suffices to check that, for every field $K$ and every morphism $h\colon\spec K\to H$, there exists an extension $K^{\le\prime}$ of $K$ and a morphism $g\colon\spec K^{\le\prime}\to G$ such that $q\circ g=h\circ j$, where $j\colon\spec K^{\le\prime}\to\spec K$ is the canonical morphism (see \cite[Proposition 3.6.2, p.~244]{ega1}). Since  $h_{\lbe S}(q)$ is locally surjective, there exist a  $\tau$ covering $\{u_{\alpha}\colon Y_{\alpha}\to\spec K\}$ in $\mathcal C$ and $S$-morphisms $g_{\alpha}\colon Y_{\alpha}\to G$ such that $q\circ g_{\alpha}=h\circ u_{\alpha}$ for every $\alpha$. Now, since $\coprod Y_{\alpha}\to\spec K$ is surjective, there exists an index $\alpha_{0}$ such that $Y_{\alpha_{0}}\neq\emptyset$. Choose $y\in Y_{\alpha_{0}}$ and let $i\colon \spec\kappa(y)\to Y_{\alpha_{0}}$ be the canonical morphism. Then $K^{\le\prime}=\kappa(y)$ and $g=g_{\alpha_{\le 0}}\circ i\colon\spec K^{\le\prime}\to G$ are the required extension and morphism, respectively.
\end{proof}

\begin{lemma}\label{tau-ex} Let $q\colon G\to H$ be a morphism of groups in $\mathcal C$.
If $q$ is an fppf (respectively, fpqc) morphism, then the sequence
\[
1\to \krn q\to G\overset{\!q}\to
H\to 1
\]
is exact for the fppf (respectively, fpqc) topology on $\mathcal C$. 
\end{lemma}
\begin{proof} It suffices to check that $h_{\lbe S}(q)$ is locally surjective in the $\fppf$ (respectively, $\fpqc$) topology. Let $u\colon Z \to H$ be a morphism in $\mathcal C$.
If $q$ is an fppf (respectively, fpqc) morphism, then $q\lbe\times_{H}\lbe Z\colon G\times_{H}Z\to Z$ is an fppf (respectively, fpqc) morphism. Thus $q\times_{H}Z$ is an fppf (respectively, fpqc) covering and $q\circ{\rm{pr}}_{1}=u\circ(q\times_{H}Z\e)$, where ${\rm{pr}}_{1}\colon  G\times_{H}Z\to G$ is the first projection. The lemma is now clear.
\end{proof}

In the next two statements, $\kbar$ is a fixed algebraic closure of a field $k$.

\begin{proposition}\label{fppf} Let $k$ be a field and let $q\colon G\to H$ be a flat morphism of $k$-group schemes locally of finite type. 
\begin{enumerate}
\item[(i)] If $q$ is surjective, then $q$ is an fppf morphism.
\item[(ii)] If $q\big(\e\kbar\e\big)\colon G\lbe\big(\e\kbar\e\big)\to
H\lbe\big(\e\kbar\e\big)$ is surjective, then $q$ is surjective.
\end{enumerate}
\end{proposition}
\begin{proof} Since $G$ is locally of finite type, $q$ is locally of finite presentation by \cite[6.2.1.2 and Proposition 6.2.3(v), p.~298]{ega1} and assertion (i) is clear. If $G$ and $H$ are of finite type, assertion (ii) is well-known \cite[I, \S3, Corollary 6.10, p.~96]{dg}. To prove (ii) when $G$ and $H$ are only locally of finite type, we may assume that $k=\kbar$ by \cite[Proposition~3.6.4, p.~245]{ega1}. 
Let $G^{\e 0}$ denote the identity component of $G$.
Since $q$ is an open morphism, $q(G^{\e 0})$ is open in $H$. On the other hand, since both $G^{\e  0}$ and $H^{\e 0}$ are of finite type by \cite[${\rm VI_{A}}$, Proposition 2.4(ii)]{sga3}, the induced morphism $q^{\le 0}\colon G^{\e 0}\to H^{\e 0}$ is quasi-compact. Consequently, $q(G^{\e 0})=q^{\le 0}(G^{\e 0}\le)$ is closed in $H$ by \cite[${\rm VI_{A}}$, Proposition 2.5.2(a)]{sga3}. We conclude that $q(G^{\le 0}\le)$ is both open and closed in $H$. Now, if $C$ is a connected component of $G$, then $C=x\le G^{\e 0}$ for some $x\in G(k)$ \cite[${\rm{VI}}_{\rm A}$, proof of Corollary 2.4.1]{sga3}, whence $q(C\le)$ is open and closed in $H$. Now let $y\in H$ and let $C^{\e\prime}$ be the connected component of $H$ which contains $y$. Since $H(k)$ is very dense in $H$ \cite[$\text{IV}_{3}$, Corollary 10.4.8]{ega}, there exists a point $z$ in $C^{\e\prime}\cap H(k)= C^{\e\prime}\cap q(G(k))$. Let $x\in G(k)$ be such that $z=q(x)$. If $C$ is the connected component of $G$ which contains $x$, then $z\in  C^{\e\prime}\cap q(C)$, whence $C^{\e\prime}\cap q(C)\neq\emptyset$. Thus $C^{\e\prime}\subseteq q(C)$ since $q(C)$ is open and closed in $H$. Consequently, $y\in q(C)\subseteq q(G)$, which completes the proof.
\end{proof}

\begin{corollary}\label{ff} Let $k$ be a field and let $q\colon G\to H$ be a flat morphism of $k$-group schemes locally of finite type. Assume that $q\big(\e\kbar\e\big)\colon G\lbe\big(\e\kbar\e\big)\to H\lbe\big(\e\kbar\e\big)$ is surjective. 
Then the sequence of $k$-group schemes
\[
1\to \krn q\to G\overset{\!q}\to
H\to 1
\]
is exact for both the fppf and fpqc topologies on $(\mr{Sch}/k)$.
\end{corollary}
\begin{proof} This is immediate from the proposition and Lemma \ref{tau-ex}.
\end{proof}

\section{Projective limits of schemes}\label{plim-sec}
To our knowledge, the results presented in this Section do not appear in the standard literature on projective limits of schemes (e.g., in \cite[$\text{IV}_{3}$, \S8]{ega}). These results will be used in Section \ref{perf} to extend to the category of schemes the ring-theoretic constructions of Section \ref{pclos}.

In this Section, $k$ denotes an arbitrary field.

If $S$ is a scheme, $\Lambda$ is a directed set and $(X_{\lbe\lambda},u_{\lambda,\e\mu}\e ; \lambda,\e\mu\in\Lambda)$ is a projective system of $S$-schemes with affine transition morphisms, then $X=\varprojlim X_{\lbe\lambda}$ exists in the category of $S$-schemes. Further, for every $S$-scheme $Z$, there exists a canonical bijection 
\begin{equation}\label{plim}
\mathrm{Hom}_{\le S}\big(Z,\varprojlim X_{\lbe\lambda}\big)=\varprojlim\mathrm{Hom}_{\le S}(Z,X_{\lbe\lambda}).
\end{equation}
See \cite[$\text{IV}_{3}$, Proposition 8.2.3 and Lemma 8.2.4]{ega}.

\begin{proposition}\label{lim-prop} Let $(X_{n})$ and $(\e Y_{n})$ be projective systems of $k$-schemes with index set $\N$ and affine transition morphisms and let  $X=\varprojlim X_{\lbe n}$ and $Y=\varprojlim Y_{\lbe n}$ be the corresponding limits in the category of $k$-schemes. Further, let $(f_{n})\colon (X_{n})\to (\e Y_{n})$ be a morphism of projective systems and let $f=\varprojlim f_{n}\colon X\to Y$ be its limit. Consider, for a morphism of $k$-schemes, the property of being:
\begin{enumerate}
\item[(i)] quasi-compact;
\item[(ii)]  quasi-separated;
\item[(iii)] separated;
\item[(iv)] affine;
\item[(v)] a closed immersion;
\item[(vi)] flat.
\end{enumerate}
If $\bm{P}$ denotes one of above properties and there exists an $n_{\le 0}\in \N$ such that the $k$-morphisms $f_{n}\colon X_{n}\to Y_{n}$ have property $\bm{P}$ for all $n\geq n_{\le 0}$, then $f\colon X\to Y$ has property $\bm{P}$ as well.
\end{proposition}
\begin{proof} For every  $n\in\N$, there exists a commutative diagram
\[
\xymatrix{X\ar[d]\ar[r]^{f}& Y\ar[d]\\
X_{n}\ar[r]^{f_{n}}&
Y_{n},}
\]
where the vertical morphisms are affine \cite[$\text{IV}_{3}$ (8.2.2)]{ega} and therefore quasi-compact and separated. Consequently, if a {\it single} morphism $f_{n}$ has one of the properties (i)-(iv), then $f$ has the same property by \cite[II, Proposition 1.6.2]{ega} and \cite[Propositions 5.3.1(v), p.~279, 6.1.5(v), p.~291, and 6.1.9(v), p.~294]{ega1}. For (v),
we may localize \cite[Corollary 4.2.4(ii), p.~262]{ega1} to reduce to the case $(f_{n})=(\spec\phi_{n})$, where $(\phi_{n})\colon(B_{n})\to(A_{n})$ is a morphism of direct systems of $k$-algebras and the maps $\phi_{n}\colon B_{n}\to A_{n}$ are surjective homomorphisms of $k$-algebras for every $n\geq n_{\le 0}$ \cite[Remark 4.2.1.1, p.~260]{ega1}. Since the direct limit functor is exact on the category of $k$-algebras, the homomorphism of $k$-algebras $\varinjlim\phi_{n}\colon \varinjlim B_{n}\to
\varinjlim A_{n}$ is surjective. This completes the proof for property (v). In the case of property (vi), the proposition follows (after localizing, as above) from \cite[I, \S2, no.~7, Proposition 9, p.~20]{bou}.
\end{proof}

\smallskip

\begin{remarks}\label{lim-rems}\indent
\begin{enumerate}
\item[(a)] In the setting of the proposition, if $(X_{n})$ and $(\e Y_{n}\lbe)$  are projective systems of $k$-group schemes (in particular, their transition morphisms are morphisms of $k$-group schemes, i.e., homomorphisms), then $f\colon X\to Y$ is a morphism of $k$-group schemes. This follows from \eqref{plim}.

\item[(b)] Clearly, properties of a morphism whose definition includes the condition of being (locally) of finite type (e.g., proper, smooth, quasi-projective, etc.) are not preserved in the limit.
\end{enumerate}
\end{remarks}

In general, open immersions are not preserved in the limit, as Example \ref{exm} below shows. However, open and arbitrary immersions are preserved in the limit in a particular case that will be relevant in the next section.

\begin{proposition}\label{lim-prop2}  In the setting of Proposition {\rm \ref{lim-prop}}, assume in addition that the transition maps of the systems $(X_{n})$ and $(\e Y_{n})$ are homeomorphisms. If, for some $n_{\le 0}\in\N$, $f_{n}\colon X_{n}\to Y_{n}$ is an immersion (respectively, open immersion) for every $n\geq n_{\le 0}$, then the limit morphism $f\colon X\to Y$ is an immersion (respectively, open immersion).
\end{proposition}
\begin{proof} By \cite[$\text{IV}_{3}$, Proposition 8.2.9]{ega}, we may identify $|X|$ and $|X_{n}|$ for every $n\in\N$, and similarly for $Y$. Under these identifications, $|f|\colon|X|\to|Y|$ is identified with $|f_{n}|\colon|X_{n}|\to|Y_{n}|$ for every $n\in\N$. Thus, by \cite[Proposition 4.2.2, p.~260]{ega1}, it suffices to check that $f$ induces a surjection (respectively, bijection) $\mathcal O_{\le Y,f(x)}\to \mathcal O_{X,x}$ for every $x\in X$. By \cite[$\text{IV}_{3}$, Corollary 8.2.12]{ega} and the identifications made above, there exists a canonical isomorphism
\[
\mathcal O_{X}=\varinjlim_{\e n\e\geq\e n_{\le 0}}\!\psi_{n}^{*}(\mathcal O_{\lbe X_{n}}),
\]
where $\psi_{n}\colon X\to X_{n}$ is the canonical projection, and similarly for $Y$.
Thus, since $\mathcal O_{\e Y_{n},\e f(x)}\to \mathcal O_{X_{n},\e x}$ is a surjection (respectively, bijection) for every $n\geq n_{\le 0}$, the direct limit morphism $\mathcal O_{\le Y,\e f(x)}\to \mathcal O_{X,\e x}$ has the same properties.
\end{proof}

\begin{example}\label{exm} In general, as noted above, open immersions are not preserved in the limit, as the following example shows (compare with \cite[$\text{IV}_{3}$, comment preceding (8.3.9)]{ega}). Let $A$ be an integral domain which is not a field and let $K$ be its field of fractions. Note that $K$ is the inductive limit of the localizations $A_{\lbe f}$ as $f$ ranges over the set $A\setminus\{0\}$ ordered by divisibility. Set $X_{\be f}=\spec A_{\lbe f}$ and $Y_{\be f}=\spec A$ for every $f\in A\setminus\{0\}$ and consider the morphism of projective systems (with evident transition morphisms) $(X_{\be f})\to (\e Y_{\be f})$, where each map $X_{\be f}\to Y_{\be f}$ is the canonical open immersion $\spec A_{\lbe f}\to\spec A$. Then the limit morphism is the canonical morphism $\spec K\to\spec A$, which is not, in general, an open immersion. To obtain an example with index set $\N$, choose $A=\Q[\e t_{m}\e;\e m\in\N\e]$ and, for every $n\in\N$, set $f_{n}=\prod_{\le m=1}^{\e n}\! t_{m}$. Then, choosing $X_{\lbe f_{n}}=\spec A_{\lbe f_{n}}$ and $Y_{\be f_{n}}=\spec A$ as above, we obtain the desired example.
\end{example}

\begin{proposition}\label{krull0} Let $(X_{n})$ be a projective system of finite $k$-schemes with index set $\N$. Then $X=\varprojlim X_{n}$ is an affine $k$-scheme of dimension zero and every residue field of $X$ is an algebraic extension of $k$.
\end{proposition}
\begin{proof} For every $n\in\N$, set $X_{n}=\spec A_{\le n}$ and let $A=\varinjlim A_{\le n}$, so that $X=\spec A$. Since each $k$-algebra $A_{\le n}$ is integral over $k$, $A$ is integral over $k$ as well. Thus, by \cite[Proposition 5.10(a), p.~69]{liu}, $\dim X=\dim A=0$. Finally, if $\mathfrak m$ is a maximal ideal of $A$, then $A/\mathfrak m$ is integral and therefore algebraic over $k$.
\end{proof}

The next statement concerning abelian groups is well-known. We sketch its proof in order to adapt it to the setting of commutative $k$-group schemes in Proposition \ref{lim-surj} below.

\begin{lemma}\label{abgps} Let
\[
0\to (F_{n})\to (G_{n})\to (H_{n})\to 0
\]
be an exact sequence of projective systems of abelian groups with index set $\N$. If
the transition morphisms of the system $(F_{n})$ are surjective, then the sequence of abelian groups
\[
0\to \varprojlim F_{n}\to \varprojlim G_{n}\to \varprojlim H_{n}\to 0
\]
is exact.
\end{lemma}
\begin{proof} (Sketch. For more details, see \cite[Lemma 7.2.8, p.~175]{ega1}).
For every $n\in\N$, let $f_{n}\colon G_{n}\to H_{n}$ be the given homomorphism. It is well-known that the nontrivial part of the present proof is to establish the surjectivity of $f=\varprojlim f_{n}\colon \varprojlim G_{n}\to \varprojlim H_{n}$. For every $n\geq 2$, let $u_{n}\colon F_{n}\to F_{n-1}$, $v_{n}\colon G_{n}\to G_{n-1}$ and $w_{n}\colon H_{n}\to H_{n-1}$ be the given transition morphisms. Let $\varphi=\varprojlim \varphi_{n}\in H$, where $(\varphi_{n})\in (H_{n})$ is a coherent sequence, i.e., $w_{n}(\varphi_{n})=\varphi_{n-1}$. Since $f_{1}$ is surjective, there exists a lifting $\psi_{\le 1}\in G_{1}$ of $\varphi_{\le 1}$, i.e., $f_{1}(\psi_{\le 1})=\varphi_{\le 1}$. Now, since $f_{2}$ is surjective, there exists a lifting $\psi_{2}^{\e\prime}\in G_{2}$ of $\varphi_{2}$, i.e., $f_{2}(\psi_{2}^{\e\prime})=\varphi_{2}$, but $\psi_{2}^{\e\prime}$ may not be coherent with $\psi_{\le 1}$, i.e., $v_{2}(\psi_{2}^{\e\prime})\neq \psi_{\le 1}$. To obtain a lifting $\psi_{2}$ of $\varphi_{2}$ which is, in fact, coherent with $\psi_{\le 1}$, we proceed as follows. Since $f_{1}(\psi_{\le 1})=\varphi_{1}=w_{2}(\varphi_{2})=w_{2}(f_{2}(\psi_{2}^{\e\prime}))=
f_{1}(v_{2}(\psi_{2}^{\e\prime}))$, the difference $\vartheta_{2}=\psi_{\le 1}-v_{2}(\psi_{2}^{\e\prime})\in\krn f_{1}=F_{1}$. Now, since $u_{2}\colon F_{2}\to F_{1}$ is surjective by hypothesis, there exists $\psi_{2}^{\e\prime\prime}\in F_{2}$ such that $u_{2}(\psi_{2}^{\e\prime\prime}\lle)=\vartheta_{2}$. Now set $\psi_{2}=\psi_{2}^{\e\prime}+\psi_{2}^{\e\prime\prime}\in G_{2}$. Then $\psi_{2}$ is a lifting of $\varphi_{2}$ which is coherent with $\psi_{\le 1}$, i.e., $f_{2}(\psi_{2})=\varphi_{2}$ and $v_{2}(\psi_{2})=\psi_{\le 1}$. We now repeat the above argument starting with $\varphi_{2}$ in place of $\varphi_{\le 1}$ to obtain a lifting $\psi_{3}$ of $\varphi_{3}$ which is coherent with $\psi_{2}$, and so forth. In this way we obtain a coherent sequence
$(\psi_{n})\in (G_{n})$ which is a lifting of $(\varphi_{n})\in (H_{n})$. Then $\psi=\varprojlim\psi_{n}\in G$ is a lifting of $\varphi\in H$, i.e., $f(\psi)=\varphi$, which completes the proof.
\end{proof}

The following statement extends the previous lemma to the category of commutative group schemes over a field $k$.
 
\begin{proposition}\label{lim-surj} Let
\[
0\to (F_{n})\to (G_{n})\to (H_{n})\to 0
\]
be a sequence of projective systems of commutative $k$-group schemes indexed by $\N$ with affine transition morphisms. Assume that the following conditions hold.
\begin{enumerate}
\item[(i)] For every $n\in\N$, the sequence of commutative $k$-group schemes
\[
0\longrightarrow F_{n}\overset{i_{n}}{\longrightarrow } G_{n}\overset{f_{\lbe n}}{\longrightarrow } H_{n}\longrightarrow  0
\]
is exact for the fpqc topology on $(\mr{Sch}/k)$.
\item[(ii)] For every $n\in\N$, $f_{n}\colon G_{n}\to H_{n}$ is flat and quasi-compact.
\item[(iii)] The transition morphisms of the system $(F_{n})$ are surjective.
\end{enumerate}
Then the sequence of commutative $k$-group schemes
\[
0\to \varprojlim F_{n}\to \varprojlim G_{n}\to \varprojlim H_{n}\to 0
\]
is exact for the fpqc topology on $(\mr{Sch}/k)$.
\end{proposition}
\begin{proof}
By (i) and Lemma \ref{surj}, $f_{n}$ is surjective for every $n\in\N$. Now, for each integer $n\geq 2$, let $u_{n}\colon F_{n}\to F_{n-1}$, $v_{n}\colon G_{n}\to G_{n-1}$ and $w_{n}\colon H_{n}\to H_{n-1}$ be the given transition morphisms. Then
$v_{n}\circ i_{n}=i_{n-1}\circ u_{n}$ for every $n\geq 2$.
Now set $F=\varprojlim F_{n},G=\varprojlim G_{n},H=\varprojlim H_{n}$ and $f=\varprojlim f_{n}\colon G\to H$. By (i), \eqref{plim} and the left-exactness of the inverse limit functor on the category of abelian groups, the sequence
\[
0\to F\to G\overset{f}\to H
\]
is exact for the fpqc topology on $(\mr{Sch}/k)$. Further, by (ii) and Proposition \ref{lim-prop}, $f$ is flat and quasi-compact.  Thus, by Lemma \ref{tau-ex}, we are reduced to checking that $f$ is a surjective morphism of schemes. 
Recall that, by \cite[Proposition 3.6.2, p.~244]{ega1}, $f$ is surjective if, and only if, for every field $K$ and every morphism $\varphi\colon\spec K\to H$, there exist an extension $K^{\lle\prime}$ of $K$ and a morphism $\psi\colon\spec K^{\le\prime}\to G$ such that the following diagram commutes
\begin{equation}\label{f-diag}
\xymatrix{\spec K'\ar@{-->}[d]_{\psi}\ar@{-->}[r]& \spec K\ar[d]^{\varphi}\\
G\ar[r]^{f}& H.}
\end{equation}
We will construct $\psi$ and $K'\e$ by adapting to the present context the proof of Lemma \ref{abgps}.

For every $n\in\N$, let $\varphi_{n}\colon \spec K\to H_{n}$ be the composition of $\varphi\colon \spec K\to H$ and the canonical morphism $H\to H_{n}$. Then $(\varphi_{n})$ is a coherent sequence, i.e., $w_{n}\circ \varphi_{n}=\varphi_{\le n-1}$ for every $n\geq 2$, and $\varphi=\varprojlim \varphi_{n}$ by \eqref{plim}. Now, since $f_{1}\colon G_{1}\to H_{1}$ is surjective, there exist an extension $K_{1}$ of $K$ and a morphism $\chi_{1}\colon \spec K_{1}\to G_{1}$ such that the following diagram commutes:
\begin{equation}\label{step0}
\xymatrix{\spec K_{1}\ar@{-->}[d]_{\chi_{1}}\ar@{-->}[r]^{g_{\le 1}}& \spec K\ar[d]^{\varphi_{\le 1}}\\
G_{1}\ar[r]^{f_{1}}& H_{1}.}
\end{equation}
Next, since $f_{2}$ is surjective, there exist an extension $K_{2}^{\e\prime}$ of $K_{1}$ and a morphism $\chi_{2}^{\e\prime}\colon \spec K_{2}^{\e\prime}\to G_{2}$ such that the following diagram commutes:
\[
\xymatrix{\spec K_{2}^{\e\prime}\ar@{-->}[d]_{\chi_{2}^{\e\prime}}\ar@{-->}[r]^{g_{2}^{\prime}}& \spec K_{1}\ar[d]^{\varphi_{2}\le \circ\e g_{\le 1}}\\
G_{2}\ar[r]^{f_{2}}& H_{2}.}
\]
We have
\[
f_{1}\circ\chi_{1}\circ g_{2}^{\le \prime}=\varphi_{\le 1}\circ g_{\le 1}\circ g_{2}^{\le \prime}=w_{2}\circ \varphi_{2}\circ g_{\le 1}\circ g_{2}^{\le \prime}=w_{2}\circ f_{2}
\circ \chi_{2}^{\e\prime}=f_{1}\circ v_{2}\circ \chi_{2}^{\e\prime}.
\]
Thus $\chi_{1}\circ g_{2}^{\le\prime}-v_{2}\circ \chi_{2}^{\e\prime}\colon \spec K_{2}^{\e\prime}\to G_{1}$ factors through a morphism $\vartheta_{2}\colon \spec K_{2}^{\e\prime}\to F_{1}$, i.e., 
\begin{equation}\label{eqtn}
\chi_{1}\circ g_{2}^{\le \prime}=v_{2}\circ \chi_{2}^{\e\prime}+i_{1}\circ\vartheta_{2}.
\end{equation}
Now, since $u_{2}\colon F_{2}\to F_{1}$ is surjective by (iii), there exist an extension $K_{2}$ of $K_{2}^{\le\prime}$ and a morphism $\chi_{2}^{\e\prime\prime}\colon 
\spec K_{2}\to F_{2}$ such that the following diagram commutes
\[
\xymatrix{\spec K_{2}\ar@{-->}[d]_{\chi_{2}^{\e\prime\prime}}\ar@{-->}[r]^(.50){g_{2}^{\e\prime\prime}}&
\spec K_{2}^{\e\prime}\ar[d]^{\vartheta_{2}}\\
F_{2}\ar[r]^{u_{2}}& F_{1}.}
\]
Set $\chi_{2}= \chi_{2}^{\e\prime}\circ g_{2}^{\e\prime\prime}+i_{2}\circ \chi_{2}^{\e\prime\prime}\colon\spec K_{2}\to G_{2}$ and $g_{\le 2}=g_{\le 1}\circ g_{2}^{\e\prime}\circ g_{2}^{\e\prime\prime}\colon \spec K_{2}\to \spec K$. Since $f_{2}\circ i_{2}\colon F_{2}\to H_{2}$ is the zero morphism (i.e., $f_{2}\circ i_{2}$ factors through the zero section $\spec k\to H_{2}$), we have $f_{2}\circ\chi_{2}=f_{2}\circ\chi_{2}^{\e\prime}\circ g_{2}^{\e\prime\prime}=
\varphi_{2}\circ g_{\le 1}\circ g_{2}^{\e\prime}\circ g_{2}^{\e\prime\prime}=\varphi_{2}\circ g_{\le 2}$, i.e., the following diagram commutes:
\begin{equation}\label{step1}
\xymatrix{\spec K_{2}\ar@{-->}[d]_{\chi_{2}}\ar@{-->}[r]^(.50){g_{\le 2}}&
\spec K\ar[d]^{\varphi_{2}}\\
G_{2}\ar[r]^{f_{2}}& H_{2}.}
\end{equation}
Further, by \eqref{eqtn}, we have
\begin{align*}
\chi_{1}\circ g_{2}^{\e\prime}\circ g_{2}^{\e\prime\prime}&=v_{2}\circ \chi_{2}^{\e\prime}\circ g_{2}^{\e\prime\prime}+i_{1}\circ\vartheta_{2}\circ g_{2}^{\e\prime\prime}=v_{2}\circ \chi_{2}^{\e\prime}\circ g_{2}^{\e\prime\prime}+i_{1}\circ u_{2}\circ \chi_{2}^{\e\prime\prime}\\
&=v_{2}\circ \chi_{2}^{\e\prime}\circ g_{2}^{\e\prime\prime}+v_{2}\circ i_{2}\circ\chi_{2}^{\e\prime\prime}=v_{2}\circ\chi_{2},
\end{align*}
i.e., the following diagram commutes
\begin{equation}\label{coh1}
\xymatrix{\spec K_{2}\ar[d]_{\chi_{2}}\ar[rr]^{g_{2}^{\e\prime}\circ g_{2}^{\e\prime\prime}}&&\spec K_{1}\ar[d]^{\chi_{1}}\\
G_{2}\ar[rr]^{u_{2}}&& G_{1}.}
\end{equation}
We now repeat the above argument starting with diagram \eqref{step1} in place of diagram \eqref{step0} and obtain diagrams similar to \eqref{step1} and \eqref{coh1} associated to $\varphi_{3}$, and so forth. In this way we obtain field extensions $K\subseteq K_{1}\subseteq K_{2}\subseteq\dots$ and morphisms $\chi_{n}\colon \spec K_{n}\to G_{n}$ for $n=1,2,\dots$, such that, for every $n\in\N$, the following diagrams commute:
\[
\xymatrix{\spec K_{n}\ar[d]_{\chi_{n}}\ar[r]& \spec K\ar[d]^{\varphi_{n}}\\
G_{n}\ar[r]^{f_{n}}& H_{n}}
\]
and
\[
\xymatrix{\spec K_{n+1}\ar[d]_{\chi_{n+1}}\ar[r]& \spec K_{n}\ar[d]^{\chi_{n}}\\
G_{n+1}\ar[r]^{u_{n+1}}& G_{n}.}
\]
Let $K^{\le\prime}=\bigcup_{n\le\geq\le 1} K_{n}$ and, for every $n\in\N$, let
$\psi_{n}\colon\spec K^{\le\prime}\to G_{n}$ be the composition of $\chi_{n}\colon\spec K_{n}\to G_{n}$ and the canonical morphism $\spec K^{\le\prime}\to\spec K_{n}$. Then the sequence $(\psi_{n})$ is a coherent lifting of $(\varphi_{n})$, i.e., for every $n\in\N$ we have $u_{n+1}\circ\psi_{n+1}=\psi_{n}$ and the following diagram commutes:
\[
\xymatrix{\spec K^{\le\prime}\ar[d]_{\psi_{n}}\ar[r]& \spec K\ar[d]^{\varphi_{n}}\\
G_{n}\ar[r]^{f_{n}}& H_{n}.}
\]
The field $K^{\le\prime}$ and the morphism $\psi=\varprojlim \psi_{n}\colon \spec K^{\e\prime}\to G$ are the ones required to make \eqref{f-diag} commute. This completes the  proof.
\end{proof}

\begin{remarks}\label{gens}\indent
\begin{enumerate}
\item[(a)] As is well-known, Lemma \ref{abgps} holds, more generally, if $(F_{n})$ satisfies the Mittag-Leffler condition \cite[Proposition 9.1(b), p.~192]{hart} (recall that the Mittag-Leffler condition holds if the transition maps of the system $(F_{n})$ are surjective). Thus Proposition \ref{lim-surj} may be regarded as a partial extension of \cite[Proposition 9.1(b), p.~192]{hart} to the category of commutative $k$-group schemes. Note also that, by Remark \ref{lim-rems}(b), the fppf topology cannot be used in place of the fpqc topology in Proposition \ref{lim-surj}. For the same reason, we cannot prove Proposition \ref{lim-surj} by working only with $\kbar$-rational points since it seems unlikely that Proposition \ref{fppf}(ii) remains valid when the morphism $q$ in that statement is quasi-compact rather than locally of finite type.
\item[(b)] The statemens \ref{lim-prop}\,\e--\e\,\ref{lim-prop2} are valid in the more general setting of \cite[$\text{IV}_{3}$, \S8.2]{ega}, i.e., when $\spec k$ is replaced by an arbitrary scheme and $\N$ is replaced by an arbitrary directed set.
\end{enumerate}
\end{remarks}

\section{The perfect  closure of a ring}\label{pclos}
 
Let $p$ be a prime number and let $\mathbb F_{\be p}$ be the field with $p$ elements. If $A$ is a ring of characteristic $p$, i.e, an $\mathbb F_{\be p}$-algebra, let $F_{\!\lbe A}$ denote the Frobenius endomorphism of $A$, i.e., $F_{\!\lbe A}(a)=a^{\le p}$ for every $a\in A$. The $\mathbb F_{\be p}$-algebra $A$ is said to be {\it perfect} if $F_{\!\be A}$ is an isomorphism.  For every $\mathbb F_{\be p}$-algebra $A$, set $F^{\e n}_{\!\be A}=F_{\!\be A}\circ F_{\!\be A}\circ\dots\circ F_{\!\be A}$ ($n$ times) if $n>0$ and let $F^{\le n}_{\!\be A}$  be the identity morphism of $A$ if $n=0$. If $A$ is perfect and $n<0$, set $F^{\e n}_{\!\be A}=\big(F^{\e -1}_{\!\be A}\big)^{\be-n}$. For every $n$ such that $F_{\!\be A}^{\le n}$ is defined and every $x\in A$, we will write $x^{\lle p^{n}}=F_{\!\be A}^{\le\lle n}(x)$.

The {\it perfect closure} of the $\mathbb F_{\be p}$-algebra $A$ is a pair $(A^{\pf},\phi_{\lbe A})$ consisting of a perfect $\mathbb F_{\be p}$-algebra $A^{\pf}$ and a homomorphism of $\mathbb F_{\be p}$-algebras $\phi_{\be A}\colon A\to A^{\pf}$ which has the following universal property: for every perfect $\mathbb F_{\be p}$-algebra $B$ and every homomorphism of $\mathbb F_{\be p}$-algebras $\psi\colon A\to B$, there exists a unique homomorphism of $\mathbb F_{\be p}$-algebras $\psi^{\e\pf}\colon A^{\pf}\to B$ such that $\psi^{\e\pf}\circ\phi_{\be A}=\psi$, i.e., the following diagram commutes
\begin{equation}\label{tri}
\xymatrix{A\ar[dr]_(.4){\phi_{\be A}}\ar[rr]^{\psi}&& B\,,\\
& A^{\lle\pf}\ar@{..>}[ur]_{\psi^{\le\pf}}
}
\end{equation}
In other words, the assignment $A\mapsto A^{\pf}$ defines a covariant functor from the category of $\mathbb F_{\be p}$-algebras to the category of perfect $\mathbb F_{\be p}$-algebras which is left adjoint to the inclusion functor, i.e., for every perfect $\mathbb F_{\be p}$-algebra $B$, there exists a canonical bijection
\begin{equation}\label{lab}
\Hom_{\e \mathbb F_{\be\lle p}\le\text{-alg}}(A,B)\overset{\!\sim}{\to}
\Hom_{\e\text{Perf-}\mathbb F_{\be p}\text{-alg}}(A^{\pf},B),\psi\mapsto\psi^{\e\pf},
\end{equation}
whose inverse is given by 
\begin{equation}\label{pf-alg0}
\Hom_{\e\text{Perf-}\mathbb F_{\be p}\text{-alg}}(A^{\pf},B)\overset{\!\sim}{\to}
\Hom_{\e \mathbb F_{\be\lle p}\le\text{-alg}}(A,B), \omega\mapsto\omega\circ\phi_{\be A}.
\end{equation}
We will show below that $(A^{\pf},\phi_{\lbe A})$ exists. The above universal property will then show that $(A^{\pf},\phi_{\lbe A})$ is unique up to a unique isomorphism.

Let $C$ be any $\mathbb F_{\be p}$-algebra (e.g., $C=\mathbb F_{\be p}$) and let $A$ be a $C$-algebra with structural morphism $\alpha\colon C\to A$. Set $\alpha^{\pf}=(\phi_{\be A}\circ \alpha)^{\pf}$. Thus the following diagram commutes:
\[\xymatrix{C\ar[r]^{\alpha}\ar[d]^{\phi_{C}}& A\ar[d]^{\phi_{\be A}}\\
C^{\le\pf}\ar[r]^{\alpha^{\pf}}& A^{\pf}. }
\]
Then $A^{\pf}$ is a $C$-algebra via the composite homomorphism 
$\alpha^{\pf}\circ \phi_{C}=\phi_{\be A}\circ\alpha$. Further, $\phi_{\be A}\colon A\to A^{\pf}$ is a homomorphism of $C$-algebras such that, if $\psi\colon A\to B$ is a homomorphism of $C$-algebras, where $B$ is perfect, then $\psi^{\le\pf}\colon A^{\pf}\to B$ is the unique homomorphism of $C^{\e\pf}$-algebras such that $\psi^{\e\pf}\circ\phi_{\be A}=\psi$. Thus there exists a canonical bijection
\begin{equation}\label{lab2}
\Hom_{\e C\text{-\e alg}}(A,B)\overset{\!\sim}{\to}
\Hom_{\e \text{Perf-}C^{\lle \pf}\text{-\e alg}}(A^{\pf},B),\psi\mapsto\psi^{\e\pf},
\end{equation}
(which generalizes \eqref{lab}) whose inverse is given by
\begin{equation}\label{lab3}
\Hom_{\e \text{Perf-}C^{\lle \pf}\text{-\e alg}}(A^{\pf},B)\overset{\!\sim}{\to}
\Hom_{\e C\text{-\e alg}}(A,B), \, \omega\mapsto \omega\circ \phi_{\be A}.
\end{equation}
(which generalizes \eqref{pf-alg0}). In particular, if $C$ is perfect, then there exists a canonical bijection
\begin{equation}\label{cab}
\Hom_{\e \text{Perf-}C\text{-\e alg}}(A^{\pf},B)\overset{\!\sim}{\to}
\Hom_{\e C\text{-\e alg}}(A,B), \, \omega\mapsto \omega\circ \phi_{\be A},
\end{equation}
which agrees with \eqref{pf-alg0} when $C=\mathbb F_{\be p}$.
The pair $(A^{\lle\pf},\phi_{\lbe A})$ can be constructed as follows (see \cite[p.~314]{gre} or \cite[V, \S1, no.~4, pp.~A.V.5-6]{bou2}): set
\begin{equation}\label{pres}
A^{\pf}=\varinjlim_{n\,\geq\e 0} A_{n},
\end{equation}
where $A_{n}=A$ for every  $n\e\geq\e 0$ and each transition map $A_{n}\to A_{n+1}$ is the Frobenius endomorphism $F_{\!\lbe  A}$, and let $\phi_{\be A}$ be the canonical homomorphism $A=A_{0}\to A^{\pf}$. Then \eqref{pres} is a perfect $C^{\e\pf}$-algebra with the required universal property \eqref{tri}. Further, the Frobenius automorphism $F_{\!\be   A^{\pf}}\colon A^{\pf}\to A^{\pf}$ can be described as follows: if $\alpha\in A^{\pf}=\varinjlim A_{n}$ is represented by $a\in A_{\le n}$, where $n\geq 1$, then $F_{\!\be  A^{\pf}}(\alpha)$ is represented by $a$ regarded as an element of $A_{\le n-1}$. On the other hand, $F_{\!\be A^{\pf}}^{-1}(\alpha)$ is represented by $a$ regarded as an element of $A_{\le n+1}$. See \cite[p.~314, last paragraph]{gre}.

\smallskip

\begin{remark}
The perfect closure $A^{\pf}$ of an $\mathbb F_{\be p}$-algebra $A$ should not be confused with the {\it perfection} $A^{\rm perf}$ of $A$, which is defined as the {\it projective} (rather than inductive) limit $\varprojlim A$, where the transition maps are all equal to the Frobenius endomorphism $F_{\be A}$.  The latter object is relevant, for example, in the constructions of Fontaine's rings of periods \cite[1.2.2]{fon} and of the tilting functor in the theory of perfectoid algebras \cite[Theorem 5.17]{sch}.
The assignment $A\mapsto A^{\rm perf}$ defines a covariant functor from the category of $\mathbb F_{\be p}$-algebras to the category of perfect $\mathbb F_{\be p}$-algebras which is {\it right} (rather than left) adjoint to the inclusion functor. To distinguish $A^{\pf}$ from $A^{\rm perf}$, some authors (e.g., Kedlaya and Liu \cite[3.1.2 and 3.4.1]{kl}) call $A^{\pf}$ (respectively, $A^{\rm perf}$) the {\it direct} (respectively, {\it inverse}) {\it perfection} of $A$.
\end{remark}

Now let $k$ be a perfect field of characteristic $p$, so that $k$ is naturally an $\mathbb F_{\be p}$-algebra, and let $A$ be a $k$-algebra. Then $A^{\pf}=\varinjlim A_{\le n}$ is an $A$-algebra via $\phi_{\be A}$ and the induced $k$-algebra structure on $A^{\pf}$ can be described as follows: if $\lambda\in k$ and $\alpha\in A^{\pf}$ is represented by $a\in A_{\le n}$, then $\lambda\e\alpha$ is represented by $\lambda^{p^{n}}\be a\in A_{\le n}$.  By \eqref{cab}, the map
\begin{equation}\label{pf-alg}
\Hom_{\e\text{$k$-alg}}(A^{\pf},B)\overset{\!\sim}{\to}\Hom_{\e \text{$k$-alg}}(A,B), \, \omega\mapsto \omega\circ \phi_{\be A},
\end{equation}
is bijective.

The following lemma shows that the operations of perfect closure and tensor product are compatible up to a canonical isomorphism.

\begin{lemma}\label{pf-prod}  Let $A$ be a $k$-algebra and let 
$B$ and $C$ be $A$-algebras. Then the canonical map
\[
\left(\phi_{\lbe B}\be \otimes_{\lbe A}\be \phi_{\le C}\right)^{\pf}\colon (B\be\otimes_{\lbe A}\! C\e)^{\le\pf}\to B^{\e\pf}\!\otimes_{\lbe A^{\pf}}\! C^{\e\pf}
\]
is an isomorphism of perfect $k$-algebras.
\end{lemma}

\begin{proof} Let $D$ be an arbitrary perfect $k$-algebra. By the universal property of the tensor product \cite[Proposition 1.14, p.~5]{liu}, there exists a canonical bijection between the set of $k$-algebra homomorphisms $B\otimes_{\lbe A}C\to D$ and the set of pairs of $k$-algebra homomorphisms $B\to D$ and $C\to D$ which induce the same $k$-algebra homomorphism $A\to D$.  Thus, using \eqref{pf-alg}, we obtain canonical bijections 
\[
\begin{array}{rcl}
\Hom_{\e \text{$k$-alg}}\le((B\be\otimes_{A}\be C\e)^{\le\pf},D) &\simeq& \Hom_{\e \text{$k$-alg}}\le(B\be\otimes_{A}\be C,D)\\
&\simeq&\Hom_{\e \text{$k$-alg}}(B,D)\times_{\Hom_{\e \text{$k$-alg}}(A,\le D)}\Hom_{\e\text{$k$-alg}}(C, D)\\
&\simeq& \Hom_{\e \text{$k$-alg}}(B^{\le\pf}, D)\times_{\Hom_{\e \text{$k$-alg}}(A^{\pf}\be,\le D)}\Hom_{\e\text{$k$-alg}}(C^{\e\pf}, D)\\
&\simeq&\Hom_{\e\text{$k$-alg}}(B^{\le\pf}\be\otimes_{A^{\pf}}\be C^{\e\pf}, D),
\end{array}
\]
whence the lemma follows. 
\end{proof}

Note that, if $A$ is a $k$-algebra, the Frobenius endomorphism of $A$ is a homomorphism of $\mathbb F_{\be p}$-algebras which is not, in general, a homomorphism of $k$-algebras. Consequently, \eqref{pres} only defines $A^{\pf}$ as an inductive limit in the category of $\mathbb F_{\be p}$-algebras. In order to extend the operation of perfect closure to the category of $k$-schemes in the next Section, we need to represent $A^{\pf}$ as an inductive limit in the category of $k$-algebras. To this end, we introduce the following notions.

\smallskip

For every $\mathbb F_{\be p}$-algebra $C$ and every integer $n$ such that $F^{\le n}_{\!\lbe C}$ is defined, we will write $(C,F^{\e n}_{\!\lbe C})$ for the ring $C$ regarded as a $C$-algebra via $F^{\e n}_{\!\lbe C}$. Note that, if $M$ is a $C$-module, then the abelian group
$M\otimes_{\e C}\be(C,F^{\le n}_{\!\lbe C})$ can be endowed with two distinct $C$-module structures, namely the {\it standard $C$-module structure} defined by $c(m\otimes d)=(cm)\otimes d=m\otimes c^{\e\le p^{\lle n}}\!\lbe d$, and the {\it non-standard $C$-module structure} defined by $c(m\otimes d)=m\otimes c\le d$, where $m\in M$ and $c,d\in C$.
If $M$ is a $C$-algebra, then the preceding structures are, in fact, $C$-algebra structures.
We will write $M^{\le(F^{\lle n}_{\!\lbe C})}$ for the abelian group $M\otimes_{\e C}\be(C,F^{\le\lle n}_{\!\lbe C})$ endowed with its {\it non-standard} $C$-module structure. Thus
\begin{equation*}\label{mpn}
M^{\le(F^{\lle n}_{\!\lbe C})}= M\otimes_{\e C}\be(C,F^{\le n}_{\!\lbe C})\qquad\text{(as abelian groups)}
\end{equation*}
and the $C$-module structure on $M^{\le(F^{\lle n}_{\!\lbe C})}$ is given by $c(m\otimes d)=m\otimes c\le d\in M^{\le(F^{\lle n}_{\!\lbe C})}$ for every $c\in C$ and $m\otimes d\in M^{\le(F^{\lle n}_{\!\lbe C})}$ (where $m\in M$ and $d\in C\le$). For every pair of integers $m,n$ such that $F^{\le n}_{\!\lbe C}$ and $F^{\le m}_{\!\lbe C}$ are defined, there exists a canonical isomorphism of $C$-modules
\begin{equation}\label{idnt}
\big(M^{\le(F^{\le n}_{\!\lbe C})}\big)^{\le(F^{\le m}_{\!\lbe C})}\!\overset{\!\sim}{\to} M^{\le(F^{\le\lle n\le+m}_{\!\lbe C})},
\end{equation}
which is defined on generators by $(m\otimes c)\otimes d\mapsto m\otimes\le c^{\e p^{ m}}\! d$, where  $m\in M$ and $c,d\in C$ (its inverse is defined on generators by $m\otimes c\mapsto (m\otimes 1)\otimes c$). If $n=0$, then \eqref{idnt} is the identity map on $M^{\le(F^{\le m}_{\!\lbe C})}$.

\smallskip
 
Now let $n$ be an integer such that $F^{\lle n}_{\be\be C}$ is defined and let ${}^{F^{\lle n}_{\be\lbe C}}\!\lbe M$ denote the abelian group $M$ equipped with the new $C$-module structure defined by $c\be\cdot\be m=c^{\e\le p^{n}}\! m$, where $c\in C$ and $m\in M$. For every pair of $C$-modules $M,N$, there exist canonical bijections
\begin{equation}\label{idnt2}
\Hom_{\e \text{$C$-mod}}\lbe\big(N^{\le(F^{\lle n}_{\!\lbe C})}, M\big)\overset{\!\sim}{\to}\Hom_{\e \text{$C$-mod}}\lbe\big(N, {}^{F^{\lle n}_{\be\lbe C}}\!\lbe M\e\big), \quad f\mapsto  f\circ \lambda_{N,\le n} ,
\end{equation}
where $\lambda_{N,\le n}\colon N\to N^{\le(F^{n}_{\!\lbe C})}$ maps $x$ to $x\otimes 1$. Note that, if $M$ is a $C$-algebra, then both $M^{\le(F^{\lle n}_{\!\lbe C})}$ and ${}^{F^{ n}_{\be\be C}}\!\lbe M$ are $C$-algebras.

If $C$ is {\it perfect}, then $F^{\lle n}_{\be\lbe C}$ is defined for every $n\in\Z$ and the canonical map
\begin{equation}\label{pn}
{}^{F^{\lle n}_{\be\lbe C}}\!\be M\to M^{\le(F^{-n}_{\!\lbe C})}, \, m\mapsto m\otimes 1,
\end{equation}
is an isomorphism of $C$-modules with inverse given by $m\otimes c\mapsto c^{\e\lle p^{\le n}}\!m$, where $m\in M$ and $c\in C$.  In this case  \eqref{idnt2} and \eqref{pn} induce a bijection
\begin{equation*}\label{idnt1}
\Hom_{\e \text{$C$-mod}}\lbe\big(N^{\le(F^{n}_{\!\lbe C})},M\big)\overset{\!\sim}{\to}\Hom_{\e \text{$C$-mod}}\lbe\big(N,M^{\le(F^{\lle -n}_{\!\lbe C})}\e\big).
\end{equation*}

Now, if $M$ is a $k$-module and $n\in\Z$ is arbitrary, the $k$-modules $M^{\le(\le F^{\lle n}_{\!\lbe k})}$ and ${}^{F^{\lle n}_{\!\lbe k}}\be\be M$ will be denoted by $M^{\le(\e p^{\lle n})}$ and ${}^{p^{\lle n}}\be\be M$, respectively. If $C$ is a $k$-algebra, $M$ is a $C$-module  and $n$ is an integer such that $F^{\le n}_{\be\lbe C}$ is defined, then the $C$-module ${}^{F^{ \lle n}_{\be\lbe C}}\!\lbe M$ endowed with the $k$-module structure induced by the map $k\to C$ is canonically isomorphic to the $k$-module ${}^{p^{\lle n}}\be\be\lbe M$. Henceforth these $k$-modules will be identified. Further, we will write ${}^{p^n}\!\lbe M$ for ${}^{F^{ n}_{\be\lbe C}}\!\lbe M$. No confusion should result since it will always be clear from the context whether ${}^{p^n}\be\be\lbe M$ is being regarded as a $C$-module or as a $k$-module.

\begin{caveat} \label{log} If $C$ is a $k$-algebra and $F^{\le n}_{\be\lbe C}$ is defined, then the above identification of $k$-modules ${}^{F^{ n}_{\be\lbe C}}\!\lbe M={}^{p^n}\be\be M$ does not extend to  
$M^{\le(F^{\lle n}_{\!\lbe C})}$ and $M^{\lle(\e p^{\lle n})}$, as we now explain. Clearly the $C$-module $M^{\le(F^{\lle n}_{\!\lbe C})}$ can be regarded as a $k$-module via the structural map $\iota\colon k\to C$, which induces a homomorphism of abelian groups $\iota^{\prime}\colon M^{\le(\e p^{\lle n})}\to M^{\le(F^{\lle n}_{\!\lbe C})},\,\, m\otimes a\mapsto m\otimes \iota(a)$. However, $\iota^{\prime}$ is not in general an isomorphism of $k$-modules since it may fail to be bijective. Indeed, since $k$ is perfect, every element of $M^{\le(\e p^{\lle n})}=M\otimes_{\le k}\!(k,\le F^{\le n}_{\!\lbe k})$ can be written in the form $m\otimes 1$. This is also the case for $M^{\le(F^{\lle n}_{\!\lbe C})}$ if $C$ is perfect, whence $\iota^{\prime}$ is bijective in this case. However, if $C$ is not perfect, then $\iota^{\prime}$ may fail to be surjective, e.g., for $M=C$. Indeed, if $M=C$ and $n=1$, then $\iota^{\prime}\colon C\otimes_{k}(k,F_{\! k})\to C\otimes_{C} (C,F_{\! C})= (C,F_{\! C})$ maps $c\otimes 1$ to $c^{\e\lle p}$.
\end{caveat}

\smallskip

Now let $A$ be a $k$-algebra. By \eqref{idnt} and \eqref{pn} (respectively), for every pair of integers $m,n$ there exist canonical isomorphisms of $k$-algebras 
\begin{equation}\label{idntk}
\big(A^{\le(\, p^{\lle n})}\big)^{\le(\, p^{m})}\!\overset{\!\sim}{\to} A^{\le(\e p^{\le n\le+m})}
\end{equation}
and
\begin{equation}\label{pnk}
\jmath_{A}^{\e (m)}\,\colon {}^{p^{\le m}}\!\!\be A\overset{\!\sim}{\to} A^{\le(\, p^{\lle - m})}, \, a\mapsto a\otimes 1.
\end{equation}
Now, for every integer $n$, there exists a canonical  homomorphism  of $k$-algebras
\begin{equation}\label{trans}
A^{(\e p^{\le n+1})}\to A^{(\e p^{\le n})}
\end{equation}
which is defined on generators by $a\otimes x\mapsto a^{\e p}\be\otimes\lbe x$.  The composition of the map \eqref{idntk} for $m=1$ and \eqref{trans} is a homomorphism of $k$-algebras
\begin{equation}\label{relf}
F_{\! A^{(\e p^{\le n})}\be/k}\colon \big(\be A^{(\e p^{\le n})}\big)^{\be(\e p\le)}\to A^{(\e p^{\le n})}
\end{equation}
which is called the {\it relative Frobenius homomorphism of $A^{(\e p^{\le n})}$ over $k$}. 
We have $F_{\! A^{(\e p^{\le n})}\be/k}((a\otimes y)\otimes x)=x(a\otimes y)^{p}=a^{\lle p}\otimes xy^{\le\lle p}$ for every $a\in A$ and $x,y\in k$.

Now, for every $n\in\Z$, there exists a canonical isomorphism of {\it $\mathbb F_{\be p}$-algebras}
\begin{equation}\label{ion}
\iota_{\lbe A, \e n}\colon A^{(\e p^{n})}\overset{\!\sim}{\to}\big(\be A^{(\e p^{n})}\big)^{\be(\e p\le)}
\end{equation}
defined on generators by $\iota_{n}(a\otimes y)=(a\otimes y)\otimes 1$. Its inverse is defined on generators by $(a\otimes y)\otimes x\mapsto x^{\e p^{-1}}\be(a\otimes y)=a\otimes x^{\e p^{-1}}\!y$ ($a\in A,x,y\in k$). We have 
\begin{equation}\label{rel}
F_{\! A^{(\e p^{\le n})}\be/k}\be\circ\be \iota_{\lbe A, \e n}=F_{\be A^{(\e p^{n})}}. 
\end{equation}
In particular, $A$ is perfect (i.e., $F_{\! A}$ is an isomorphism) if, and only if, $F_{\be A/k}$ is an isomorphism. 
\medskip

Note that the composition 
\[
{}^{p^{-1}}\!\!\lbe A\  \underset{\!{}^{{}^{\sim}}}{\overset{\!\jmath_{\be A}^{\le(-1)}}{\lra}}  A ^{(\e p\le)}\underset{\!{}^{{}^{\sim}}}{\overset{\!\iota_{\! A, \e 0}^{-1}}{\lra}} A ,
\]
where $\jmath_{\be A}^{\le(-1)}$ is the isomorphism of $k$-algebras \eqref{pnk} and $\iota_{\lbe A, \e 0}^{-1}$ is the inverse of the isomorphism of $\mathbb F_{\be p}$-algebras \eqref{ion} when $n=0$, is the identity map on the underlying abelian groups. On the other hand, the composition
\[
{}^{p^{-1}}\!\!\lbe A \underset{ {}^{{}^{\sim}}}{\overset{\!\jmath_{\be A }^{\le(-1)}}{\lra}}  A^{\be(\e p\le)}\overset{\!F_{\!\lbe A\lbe/k}}{\lra} A,
\]
where $F_{\! A \lbe/k}$ is the map \eqref{relf} for $n=0$, is the homomorphism of $k$-algebras ${}^{p^{-1}}\!\!\lbe A \to A , a \mapsto a^{\le p}$. 
Similar results hold for $ A^{(\e p^{n})}$ in place of $A$.

We can now represent $A^{\pf}$ as an inductive limit in the category of $k$-algebras:

\begin{lemma}\label{rep} Let $A$ be a $k$-algebra. Then there exists a canonical isomorphism of $k$-algebras
\[
A^{\pf}=\varinjlim_{n\,\geq\e 0} A^{(\e p^{-n})},
\]
where the transition maps are the maps \eqref{trans}. 
\end{lemma}
\begin{proof} For every $n\geq 0$, the map $\zeta_{\le n}\colon A\to A^{(\e p^{\le -n})}, a\mapsto a\otimes 1$, is an isomorphism of {\it ${\mathbb F}_{\be p}$-algebras} whose inverse is defined on generators by $a\otimes x\mapsto x^{p^{\lle n}}\! a$ ($a\in A,x\in k$). Now, if $\alpha\in A^{\pf}=\varinjlim A_{n}$ is represented by $a_{n}\in A_{n}$, let $\varphi(\alpha)\in \varinjlim A^{(\e p^{\le -n})}$ be represented by $\zeta_{n}(a_{n})\in  A^{(\e p^{\le -n})}$. A straightforward verification shows that the map $\varphi\colon \varinjlim A_{n}\to \varinjlim A^{(\e p^{\le -n})}$ just defined is an isomorphism of $k$-algebras.
\end{proof}

The next lemma shows that the perfect closure of a $k$-algebra coincides with the perfect closure of its largest reduced quotient.

\begin{lemma}\label{red-pf} Let $A$ be a $k$-algebra. Then the canonical projection $A\to A_{\red}$ induces an isomorphism of $k$-algebras $A^{\pf}=(A_{\le\red})^{\pf}$. 
\end{lemma}
\begin{proof} The $p\e$-radical $A_{p^{\infty}}$ of $A$ considered in \cite[p.~314]{gre} coincides with the nilradical $\text{Nil}(A)$ of $A$. Now the arguments in \cite[4, p.~315]{gre} show that the map $A^{\pf}\to (A_{\e\red})^{\pf}$ induced by the canonical projection $A\to A_{\e\red}=A/\text{Nil}(A)$ is surjective with kernel contained in $\text{Nil}\big(\lbe A^{\pf}\e\big)$. But $\text{Nil}\big(\lbe A^{\pf}\e\big)=0$ since a perfect ring is reduced \cite[V, \S1, no.~4, p.~A.V.5]{bou2}. 
\end{proof}

\begin{remark}\label{ext} The preceding constructions remain valid if $k$ is replaced by any perfect $\mathbb F_{\be p}$-algebra $C$. In particular, for every $C$-algebra $A$, there exists a canonical  isomorphism of $C$-algebras 
\[
A^{\pf}=\varinjlim_{n\e\geq\e 0} A^{\le(F^{\lle -n}_{\!\lbe C})},
\]
where the transition maps are given by $a\otimes c\mapsto a^{p}\otimes c$, where $a\in A$ and $c\in C$.  When $C=\mathbb F_{\be p}$, the preceding isomorphism agrees with the case $C=\mathbb F_{\be p}$ of \eqref{pres} (indeed, $F_{\be C}$  is the identity morphism of $C$, $A^{\le(F^{\lle -n}_{\!\lbe C})}=A$ for every $n$ and the corresponding transition morphisms are all equal to $F_{\!\be A}$). 
\end{remark}

The perfect closure of some rings can be described explicitly. For example:

\begin{example}\label{pf-rem} Let $A$ be an integral domain endowed with an $\mathbb F_{\be p}$-algebra structure, let $L$ be the field of fractions of $A$ and let $\overline{L}$ be a fixed algebraic closure of $L$. Clearly $A$ is a subring of the perfect ring $\overline{L}$ and  $a^{\le p^{-n}}\in\overline{L}$ for every $a\in A$ and $n\geq 0$. Then
\[
A^{\pf}=\bigcup_{n\e\geq\e 0} \{a^{\le p^{-n}}\!; a\in A\}=A\e[\e a^{\le p^{-n}}\!; a\in A, n\in\N \e]\subseteq\overline{L}.
\]
See \cite[V, \S1, no.~4, Proposition 3, p.~A.V.6]{bou2}. Note that the map $\phi_{\be A}\colon A\to A^{\pf}$ is the canonical inclusion. In particular, let 
$A=k\e[\{ x_{i}\}]$ be the polynomial $k$-algebra on a (possibly infinite) family of independent indeterminates $\{x_{i}\}_{i\e\in\e I}$. Then 
\[
A^{\le\pf}=k\e[\e x_{i}^{\e p^{-n}}\!;\e i\in I, n\geq 0\e]=\varinjlim_{n\e\geq\e 0}k\e[\e x_{i}^{\le p^{-n}}\!;\e i\in I\e]
\]
inside an algebraic closure of $k\e(\{ x_{i}\})$. Since every ring $k\e[\e x_{i}^{\le p^{-n}}\!;\e i\in I\e]$ is a free $A$-module, we conclude that $A^{\le\pf}$ is flat over $A$ \cite[I, \S2, no.~3, Proposition 2, p.~14]{bou}. Further, since $A^{\le\pf}$ is integral over $A$, $A^{\le\pf}$ is, in fact, faithfully flat over $A$ \cite[Chapter 2, Theorem 3(2), p.~28, and Theorem 5(i), p.~33]{mat2}.
\end{example}

\begin{remarks}\indent
\begin{enumerate}
\item[(a)]  In connection with the preceding Example, there exist $\mathbb F_{\be p}\e$-algebras $A$ such that $A^{\pf}$ is not flat over $A$. Indeed, let $A$ be any $\mathbb F_{\be p}$-algebra such that $A_{\red}$ is perfect and $\text{Nil}(A)^{2}\neq \text{Nil}(A)$, e.g., $A=\mathbb F_{\be p}[x]/(x^{\le 2})$. Then
$A^{\pf}=(A_{\red})^{\pf}=A_{\red}$ by Lemma \ref{red-pf}, but $A_{\red}$ is not flat over $A$ (otherwise $\text{Nil}(A)\otimes_{A}A_{\red}=\text{Nil}(A)/\text{Nil}(A)^{2}$ would inject as a nonzero nilideal in $A_{\le\red}$).
\item[(b)] If $A$ is any $k$-algebra, then $A^{\pf}$ can be obtained from $A$ in two steps, the first of which is faithfully flat over $A$ but the second may fail to be flat over $A$. Indeed, fix a polynomial $k$-algebra $B=k\e[\{ x_{i}\}]$, where $\{x_{i}\}_{i\in I}$ is a (possibly infinite) family of independent indeterminates, and a surjective homomorphism of $k$-algebras $q\colon B\to A$. Set
\[
\widetilde{A}=B^{\le\pf}\be\otimes_{\e B}\be A
\]
with its natural $k$-algebra structure. Then $\big(\lbe\widetilde{A}\,\big)^{\lbe\pf}=B^{\le\pf}\be\otimes_{\e B^{\le\pf}}\be A^{\pf}= A^{\pf}$ by Lemma \ref{pf-prod}. Now, since the map $B^{\le\pf}\to \widetilde{A}$ induced by $q$ is surjective, the Frobenius endomorphism $F_{\be\widetilde{A}}\colon \widetilde{A}\to \widetilde{A}$ is surjective as well, i.e., $\widetilde{A}=\big(\lbe\widetilde{A}\,\big)^{p}$. It follows that $\big(\lbe\widetilde{A}\,\big)_{\be\red}$ is a perfect $k$-algebra, whence $\big(\lbe\widetilde{A}\,\big)_{\be\red}=\big(\lbe\widetilde{A}\,\big)^{\pf} =A^{\le\pf}$. Further, since $B^{\le\pf}$ is faithfully flat over $B$ by Example \ref{pf-rem}, $\widetilde{A}$ is faithfully flat over $A$. Thus $A^{\le\pf}$ can be obtained from $A$ in two steps, as claimed:
\[
A\hookrightarrow \widetilde{A}\twoheadrightarrow \big(\lbe\widetilde{A}\,\big)_{\be\red}=A^{\le\pf},
\]
where the first extension $\widetilde{A}/\be A$ is faithfully flat but the second one may not be flat, as noted in (a). Compare \cite[Lemma 0.1, p.~18]{lip}, where the algebra $\widetilde{A}$ was denoted by $\bar{A}$.
\item[(c)] In general, the $k$-algebra $\widetilde{A}$ considered in (b) depends on the choice of the pair $(B,q)$. For example, if $A=\mathbb F_{\be p}$ and we choose successively $(B,q)=(\mathbb F_{\be p}\le,1_{\mathbb F_{\lbe p}})$, where $1_{\mathbb F_{\lbe p}}$ is the identity map of $\mathbb F_{\be p}$, and $(B,q)=(\mathbb F_{\be p}[x],\text{ev}_{0})$, where $\text{ev}_{0}\colon \mathbb F_{\be p}[x]\to \mathbb F_{\be p}$ is the natural evaluation-at-zero map $f(x)\mapsto f(0)$, then we obtain correspondingly $\widetilde{A}=\mathbb F_{\be p}$ and $\widetilde{A}= \mathbb F_{\be p}\e[\e x^{\le p^{-n}}\!; n\le\in\le\N_{0}\e]/(x)$. Note that the latter algebra is non-reduced, which shows that $\widetilde{A}$ can be non-reduced even if $A$ is reduced.
\end{enumerate}
\end{remarks}

\section{The inverse perfection of a scheme}\label{perf}

In this Section we extend the constructions of the previous Section to the category of schemes.

\smallskip

Let $p$ be a prime number. For every $\mathbb F_{\be p}$-scheme $Y$, let $F_{\le Y}\colon Y\to Y$ denote the absolute Frobenius endomorphism of $Y$ \cite[$\text{VII}_{\text{A}}$, \S4.1]{sga3}. If  $A$ is an $\mathbb F_{\be p}$-algebra, then
\begin{equation}\label{frob-aff}
F_{\le \spec A}=\spec F_{\!\be A}.
\end{equation}
The scheme $Y$ is said to be {\it perfect} if $F_{\le Y}$ is an isomorphism. Given an $\mathbb F_{\be p}\le$-scheme $Y$, there exists a pair $(Y^{\lle\pf},\phi_{\le\le Y})$ consisting of a perfect $\mathbb F_{\be p}\le$-scheme $Y^{\lle\pf}$ and a morphism of $\mathbb F_{\be p}\le$-schemes $\phi_{\le\le Y}\colon Y^{\lle\pf}\to Y$ such that, for every perfect $\mathbb F_{\be p}\le$-scheme $Z$ and every morphism of $\mathbb F_{\be p}\le$-schemes $\psi\colon Z\to Y$, there exists a unique morphism of $\mathbb F_{\be p}\le$-schemes $\psi^{\e\pf}\colon Z\to Y^{\lle\pf}$ such that $\phi_{\le\lle Y}\be\circ\be\psi^{\e\pf}=\psi$.  Following \cite[Theorem 8.5.5(c)]{kl}, we call $Y^{\lle\pf}$ the {\it inverse perfection} of $Y$. We should note that $Y^{\lle\pf}$ is called the {\it perfection} of $Y$ in both \cite[p.~ 226]{mi2} and \cite[p.~3]{bw}, and the {\it perfect closure} of $Y$ in \cite[p.~317]{gre}, where it is denoted by $Y^{1\lbe/p^{\infty}}$. In the latter reference, $Y^{\lle\pf}$ is constructed by globalizing the functor on $\mathbb F_{\be p}\le$-algebras $A\mapsto A^{\pf}$. In particular, $(\spec A)^{\pf}=\spec A^{\pf}$ and $\phi_{\e\spec\lbe A}=\spec \phi_{\be A}$.

We will write $(\mathrm{Perf}/\le\mathbb F_{\be p})$ for the category of perfect $\mathbb F_{\be p}$-schemes.  It follows from the universal property described above  (or, alternatively, from \eqref{lab}) that the {\it inverse perfection functor}
\begin{equation}\label{pf-fun2}
(\mathrm{Sch}/\le\mathbb F_{\be p})\to(\mathrm{Perf}/\le\mathbb F_{\be p}), \quad Y\mapsto Y^{\lle\pf},
\end{equation}
is covariant and right-adjoint to the inclusion functor $(\mathrm{Perf}/\le\mathbb F_{\be p}) \to
(\mathrm{Sch}/\le\mathbb F_{\be p})$, i.e., for every perfect $\mathbb F_{p}$-scheme $Z$, there exists a canonical bijection 
\begin{equation}\label{pf}
\Hom_{\le\mr{Sch}/\lle\mathbb F_{\be p}}(Z,Y)\overset{\!\sim}{\to}\Hom_{\e
\mr{Perf}/\lle\mathbb F_{\be p}}\be\big(Z,Y^{\lle\pf}\le\big), \, \psi\mapsto\psi^{\le\pf}.
\end{equation}

\begin{remark}\label{uh}
It is shown in \cite[p.~317]{gre} that the canonical morphism $\phi_{\le\spec A}=\spec \phi_{\lbe A}\colon \spec A^{\pf}\to\spec A$ is a homeomorphism. Further, for every $\alpha\in A^{\pf}$ there exists an integer $n\geq 0$ such that $\alpha^{\e p^{n}}\!\in\img\big[\e\phi_{\lbe A}\colon A\to A^{\pf}\e\big]$ by \cite[p.~314]{gre}. Consequently, $\phi_{\e \spec A}$ is integral and radicial \cite[Proposition 3.7.1(b), p.~246]{ega1}. Thus $\phi_{\e\spec A}$ is, in fact, a universal homeomorphism \cite[$\text{IV}_{4}$, Corollary 18.12.11]{ega}.  It now follows from \cite[Proposition 3.8.2(v), p.~249]{ega1} that $\phi_{Y}\colon Y^{\pf}\to Y$ is a universal homeomorphism for every $\mathbb F_{\be p}\le$-scheme $Y$.
\end{remark}

More generally, if $W$ is any perfect $\mathbb F_{\be p}\le$-scheme (in lieu of $W=\spec \mathbb F_{\be p}$) and $Y$ is a $W$-scheme, then there exists a canonical bijection 
\begin{equation}\label{pfg}
\Hom_{\le\mr{Sch}/W}(Z,Y)\overset{\!\sim}{\to} \Hom_{\e \mr{Perf}/W}\big(Z,Y^{\lle\pf}\le\big),\, \psi\mapsto\psi^{\le\pf},
\end{equation}
whose inverse is given by 
\begin{equation}\label{pfc}
\Hom_{\e
\mr{Perf}/W}\big(Z,Y^{\lle\pf}\le\big)\overset{\!\sim}{\to}
\Hom_{\le\mr{Sch}/W}(Z,Y),\,\omega\mapsto \phi_{\le\le Y}\circ\omega. 
\end{equation}
The bijections \eqref{pfg} and \eqref{pfc} are induced by \eqref{lab2} and \eqref{cab}, respectively.
\medskip

In the following, we focus on the case $W=\spec k$, where $k$ is a perfect field of characteristic $p$.

Note that, if $Y$ is a $k$-scheme, then Lemma \ref{red-pf} yields a canonical  isomorphism of $k$-schemes
\begin{equation}\label{pf=pf-red}
Y^{\pf}=(Y_{\red})^{\pf}.
\end{equation}
The following alternative construction of $Y^{\pf}$ follows from Lemma \ref{rep}.

For every integer $n$, let $(\spec k,\spec F_{\! k}^{\e n})$ denote the scheme  $\spec k$ regarded as a $k$-scheme via $\spec F_{\! k}^{\e n}$. If $Y$ is a $k$-scheme, set 
\begin{equation}\label{ypn}
Y^{(\e p^{\le n})}=Y\times_{\spec k}(\spec k,\spec F_{\! k}^{\e n}).
\end{equation}
Then, for every $k$-scheme $Z$, \eqref{idnt2} induces a bijection
\begin{equation}\label{adjt}
\Hom_{\e\mr{Sch}/k}\lbe\big(Y^{(\e p^{-n\le})},Z\big)\simeq \Hom_{\e\mr{Sch}/k}\lbe(Y,Z^{(\e p^{\le n\le})}\big).
\end{equation}
Now, for every pair of integers $m,n$, the isomorphism of $k$-algebras \ref{idntk} induces an isomorphism of $k$-schemes
\begin{equation}\label{idntg}
Y^{(\e p^{\le n\le +m})}\!\overset{\!\sim}{\to}\big(Y^{(\e p^{\le n})}\big)^{\be(\e p^{\le m})}.
\end{equation}
Setting $m=1$ above, we obtain an isomorphism of $k$-schemes
\begin{equation}\label{ko}
Y^{(\e p^{\le n+1})}\stackrel{\sim}{\to}\big(Y^{(\e p^{\le n})}\big)^{\be(\e p)}.
\end{equation}
On the other hand, by \eqref{trans}, there exists a canonical morphism of $k$-schemes 
\begin{equation}\label{trans2}
Y^{(\e p^{\le n})}\to Y^{(\e p^{\le n+1})}.
\end{equation} 
The composition of  \eqref{ko} and \eqref{trans2} is the {\it relative Frobenius morphism of $Y^{(\e p^{\le n})}$ over $k$:}
\begin{equation}\label{rel-frob}
F_{ Y^{(\e p^{\le n})}/k}\colon Y^{(\e p^{\le n})}\longrightarrow \big(Y^{(\e p^{\le n})}\big)^{\be(\e p)}.
\end{equation}
By \eqref{rel}, we have 
\begin{equation}\label{rel2}
\mathrm{pr}_{\lbe Y^{(\e p^{\le n})}}\circ F_{ Y^{(\e p^{\le n})}\be/k}=
F_{\lbe Y^{(\le p^{\le n})}},
\end{equation}
where $\mathrm{pr}_{\lbe Y^{(\le p^{\le n})}}\colon \big(Y^{(\e p^{\le n})}\big)^{\be(\e p)}\to Y^{(\e p^{\le n})}$ is the first projection. It follows from \eqref{rel2} that $F_{ Y^{(\le p^{\le n})}/k}$ is a universal homeomorphism (see \cite[XV, \S1, Proposition 2(a), p.~445]{sga5}). Consequently, the transition morphisms $Y^{(\e p^{-n})}\to Y^{(\e p^{-n+1})}$ \eqref{trans2} of the  projective system of $k$-schemes $\big(Y^{(\e p^{\le -n})}\big)_{n\e\geq\e 0}$ are universal homeomorphisms (in particular they are affine \cite[II, (6.1.2)]{ega}). Now Lemma \ref{rep} and \eqref{pf=pf-red} show that there exist canonical isomorphisms of $k$-schemes
\begin{equation}\label{ypf}
Y^{\pf}=\varprojlim_{n\e\geq\e 0} Y^{(\e p^{-n})}=\varprojlim_{n\e\geq\e 0} (\le Y_{\be\red})^{(\e p^{-n})}.
\end{equation}

\begin{remark}\label{r-pfe} By Remark \ref{ext}, the preceding considerations extend to a relative setting where $\spec k$ is replaced by any nonempty perfect $\mathbb F_{\be p}\le$-scheme $T$. Compare \cite[pp.~226-227]{mi2}. Note also that the inverse perfection $Y^{\pf}$ of a $T$-scheme $Y$ depends only on its structure as a scheme over $\mathbb F_{\be p}\le$, whereas the presentation \eqref{ypf} depends on the structure of $Y$ as a scheme over $T$.
\end{remark}

\begin{proposition}\label{pf-morph} Consider, for a morphism of $k$-schemes, the property of being:
\begin{enumerate}
\item[(i)] separated;
\item[(ii)] quasi-compact;
\item[(iii)] quasi-separated;
\item[(iv)] a closed immersion;
\item[(v)] affine;
\item[(vi)] flat;
\item[(vii)] an open immersion;
\item[(viii)] an immersion.
\end{enumerate}
If $\bm{P}$ denotes one of the above properties and $f\colon X\to Y$ has property $\bm{P}$, then $f^{\pf}\colon X^{\pf}\to Y^{\pf}$ also has property $\bm{P}$.
\end{proposition}
\begin{proof} Properties (i)-(vi) are stable under base change. Thus, if $\bm{P}$ denotes one of these properties and $f$ has property $\bm{P}$, then the morphism $X^{(\e p^{-n})}\to Y^{(\e p^{-n})}$ induced by $f$ has property $\bm{P}$ as well for every integer $n\e\geq\e 0$. Consequently, by \eqref{ypf} and Proposition \ref{lim-prop}, $f^{\pf}$ has property $\bm{P}$. In the case of properties (vii) and (viii), the proposition follows from Proposition \ref{lim-prop2} since the transition morphisms of the systems $\big(X^{(\e p^{-n})}\big)$ and $\big(Y^{(\e p^{-n})}\big)$ are (universal) homeomorphisms.
\end{proof}

\begin{remarks}\label{top-prop} \indent
\begin{enumerate}
\item [(a)] Let $f\colon X\to Y$ be a morphism of $k$-schemes. Then, for every perfect $k$-scheme $Z$, there exists a canonical commutative diagram of sets
\[
\xymatrix{\Hom_{\e\mr{Sch}/k}(Z,X)\ar[r]\ar[d]^{\simeq}&\Hom_{\e\mr{Sch}/k}(Z,Y)\ar[d]^{\simeq}\\ 
\Hom_{\e\mr{Perf}/k}\big(Z,X^{\lle\pf}\le\big)\ar[r]&\Hom_{\e\mr{Perf}/k}\big(Z,Y^{\lle\pf}\le\big)\,,}
\]
where the vertical maps are the bijections \eqref{pfc} for $W=\spec k$ and the top (respectively, bottom) horizontal map is given by $\psi\mapsto f\circ\psi$ (respectively, $\varphi\mapsto f^{\le\pf}\circ\varphi$). It follows that, if the top horizontal map in the above diagram is a bijection for every affine perfect $k$-scheme $Z$, then $f^{\pf}\colon X^{\pf}\to Y^{\pf}$  is an isomorphism of perfect $k$-schemes.
\item[(b)] Since the canonical morphism $Y^{\pf}\to Y$ is a universal homeomorphism by Remark \ref{uh}, it is clear that if $\bm{P}$ is a purely topological property of a morphism of $k$-schemes (e.g., that of being injective, surjective, dominant, etc.) and $f\colon X\to Y$ has property $\bm{P}$, then $f^{\pf}\colon X^{\pf}\to Y^{\pf}$ also has property $\bm{P}$. Further, if $\{U_{i}\}_{i\le\in\le I}$ is an open covering of $Y$, then $\{U_{i}^{\pf}\}_{i\le\in\le I}$ is an open covering of $Y^{\pf}$. In particular, if $Y=\coprod_{\le i\in I} Y_{\be i}$, then $Y^{\pf}=\coprod_{\le i\in I} Y_{\be i}^{\pf}$.

\item[(c)] Let $f\colon X\to Y$ be an fpqc morphism of $k$-schemes, i.e., $f$ is faithfully flat and every quasi-compact open subset of $Y$ is the image of a quasi-compact open subset of $X$ \cite[Proposition 2.33, p.~27]{v}. It is clear from (b) and part (vi) of the proposition that $f^{\pf}\colon X^{\pf}\to Y^{\pf}$ is an fpqc morphism as well. 
\item[(d)] If follows from Lemma \ref{pf-prod} that, if $Y$ is a $k$-scheme and $X$ and $Z$ are $Y$-schemes, then $(X\times_{Y}Z)^{\pf}=X^{\pf}\times_{Y^{\lbe\pf}}Z^{\e\pf}$. 

\item[(e)] Let $Y$ be a $k$-group scheme. It follows from (d) and \eqref{pf} that $Y^{\pf}$ is a $k$-group scheme. On the other hand, for every integer $n$, $Y^{(\le p^{\le n})}$ is a $k$-group scheme and the relative Frobenius morphism \eqref{rel-frob} is a homomorphism \cite[${\rm VII_{A}}$, 4.1.2]{sga3}. We conclude that $\big(Y^{(\le p^{\le -n})}\big)_{n\e\geq\e 0}$ is a projective system in the category of $k$-group schemes and \eqref{ypf} is an isomorphism of $k$-group schemes. In particular, for every $m\e\geq\e 0$, the composition of \eqref{ypf} and the canonical projection morphism $\varprojlim Y^{(\e p^{\le -n})}\to Y^{(\e p^{\le -m})}$ is a homomorphism of $k$-group schemes $Y^{\pf}\to Y^{(\e p^{\le -m})}$.

\item[(f)] It follows from \cite[V, \S 1, no.~4, p.~A.V.5]{bou2} that a perfect $k$-scheme is reduced. In particular, $Y^{\pf}$ is a reduced $k$-scheme for every $k$-scheme $Y$.
\end{enumerate}
\end{remarks}

\begin{proposition}\label{fpf} Let $X$ be a $k$-scheme.
\begin{enumerate}
\item[(i)] If $X$ is \'etale over $k$, then $X$ is perfect.
\item[(ii)] If $X$ is finite over $k$, then $X^{\pf}$ is finite and \'etale over $k$.
\end{enumerate}
\end{proposition}
\begin{proof} If $X$ is \'etale over $k$, then $X\simeq \coprod_{\le i\le\in\le I}\spec k_{\le i}$ for some family $\{k_{\le i}\}$ of finite separable field extensions of $k$ \cite[$\text{IV}_{4}$, Corollary 17.4.2$({\text d}^{\prime}\e)$]{ega}. Since $k$ is perfect, each $k_{\le i}$ is perfect as well and assertion (i) follows from Remark \ref{top-prop}(b). Now assume that $X=\spec B$ is finite over $k$. Then $X_{\red}=\spec B_{\red}$ is finite  and \'etale over $k$ by \cite[V, \S6, no.~7, Lemma 5, p.~A.V.35]{bou2}. Thus, by (i) and \eqref{pf=pf-red}, $X^{\pf}=(X_{\red})^{\pf}=X_{\red}$ is finite and \'etale over $k$, as claimed.
\end{proof}

\begin{lemma}\label{pf=1} Let $G$ be a $k$-group scheme locally of finite type such that
$G\big(\e\kbar\,\big)$ is the trivial group. Then $G^{\e\pf}$  is the trivial $k$-group scheme.
\end{lemma}
\begin{proof}
By \eqref{pf=pf-red}, it suffices to check that $G_{\red}$ is the trivial $k$-group scheme. Since $k$ is perfect, $G_{\red}\times_{\spec k}\e\spec \kbar$ is reduced whence $G_{\red}\times_{\spec k}\e\spec \kbar=(G\times_{\spec k}\spec \kbar\e)_{\red}$ by \cite[Corollary 4.5.12, p.~271]{ega1}. Thus we may assume that $k=\kbar$. By \cite[${\rm VI_{A}}$, 0.2 and Lemma 0.5.2]{sga3}, $G_{\red}$ is a reduced closed $k$-subgroup scheme of $G$. Further, the hypothesis implies that $G_{\red}(k)$ is the trivial group. Now \cite[II, \S5, Proposition 4.3(a), p.~245]{dg} shows that $G_{\red}$ is the trivial $k$-group scheme.
\end{proof}

\begin{proposition}\label{lim-pf} Let $(\e Y_{\lbe\lambda}\lbe)_{\lambda\le\in\le\Lambda}$ be a projective system of $k$-schemes, where $\Lambda$ is a directed set  containing an element $\lambda_{\e 0}$ such that the transition morphisms of $k$-schemes $Y_{\be\mu}\to Y_{\lbe\lambda}$ are affine for $\mu\geq\lambda\geq \lambda_{\le 0}$. Then there exists a canonical isomorphism of perfect $k$-schemes
\[
\left(\varprojlim
Y_{\lambda}\right)^{\!\pf}=\varprojlim\be Y_{\lambda}^{\pf}.
\]
\end{proposition}
\begin{proof} By \cite[IX, \S 3, dual of Theorem 1]{mac}, we may replace $\Lambda$ with the cofinal subset $\{\lambda\in \Lambda\,|\, \lambda \geq \lambda_{\le 0}\}$ and hence assume that $\lambda_{0}$ is an initial element of $\Lambda$. For $\lambda\geq\lambda^{\prime}$, let $f_{\lambda\e,\e\lambda^{\prime}}$ denote the transition morphisms of the system $(Y_{\be\lambda})$. For each $\lambda\in\Lambda$, consider the quasi-coherent $\s O_{Y_{\lbe\lambda_{\lle 0}}}\be$-algebra $\s A_{\lambda}=f_{\lambda\e,\e\lambda_{\lle 0}*}\s O_{Y_{\lambda}}$. Then $Y_{\be\lambda}=\spec\s A_{\lambda}$ by \cite[p.~356, line -5]{ega1}. Further, $\s A=\varinjlim \s A_{\lambda}$  is a quasi-coherent $\s O_{Y_{\be\lambda_{\lle 0}}}\be$-algebra and 
$Y=\varprojlim \be Y_{\lambda}=\spec\s A$ by \cite[$\text{IV}_{3}$, Proposition 8.2.3]{ega}. Now, by \eqref{frob-aff} and the construction of $Y_{\be\lambda}=\spec\s A_{\lambda}$ in \cite[proof of Corollary 9.1.7, p.~356]{ega1}, we have $F_{Y_{\be\lambda}}=\spec F_{\!\s A_{\lambda}}$, where $F_{\!\s A_{\lambda}}$ is the absolute Frobenius endomorphism of the
$\s O_{Y_{\be\lambda_{\lle 0}}}\be$-algebra $\s A_{\lambda}$.
Similarly, $F_{Y}=\spec F_{\!\s A}$. Since $F_{\!\s A}=\varinjlim \be F_{\!\s A_{\lambda}}$, we conclude that $F_{Y}=\varprojlim\be F_{Y_{\lambda}}$. Consequently, if $Y_{\be\lambda}$ is perfect for every $\lambda$ (i.e., $F_{Y_{\be\lambda}}$ is an isomorphism), then $Y$ is perfect as well. Now, by
Proposition \ref{pf-morph}(v), $\big(\e Y_{\be\lambda}^{\pf}\e\big)$ is a projective system of 
$k$-schemes with affine transition morphisms. Thus, by the preceding discussion, its limit
$\varprojlim\be Y_{\be\lambda}^{\pf}$ is a perfect $k$-scheme. Further, if $Z$ is an arbitrary perfect $k$-scheme then, by \eqref{pf} and \eqref{plim},
\begin{align*}
\Hom_{\e \mr{Perf}/k}\!\left(Z, \varprojlim\be
Y_{\be\lambda}^{\pf}\right)=\varprojlim\Hom_{\e
\mathrm{Perf}/k}\!\left(Z,Y_{\be\lambda}^{\pf}\le\right)&=\varprojlim
\Hom_{\le\mr{Sch}/k}\!\left(Z,Y_{\be\lambda}\le\right)\\
&=\Hom_{\le
\mr{Sch}/k}\!\left(Z,\varprojlim Y_{\be\lambda}\le\right).
\end{align*}
Consequently, $\left(\varprojlim
Y_{\be\lambda}\right)^{\!\pf}=\varprojlim\be Y_{\be\lambda}^{\pf}$, as claimed.
\end{proof}

To conclude this Section, we show below that, in an appropriate sense, the inverse perfection functor commutes with Weil restriction.

Let $f\colon Z^{\prime}\to Z$ be a  morphism schemes and let $X^{\prime}$ be a  $Z^{\prime}$-scheme. We will say that {\it the Weil restriction of $X^{\prime}$ along $f$ exists} or, more concisely, that $\Re_{Z^{\prime}\be/Z}(X^{\prime}\e)$ exists, if the contravariant functor
$(\mathrm{Sch}/Z\e)\to(\mathrm{Sets}), T\mapsto\Hom_{Z^{\prime}}(T\times_{Z} Z^{\prime}, X^{\prime}\le )$, is  represented by a $Z$-scheme $\Re_{Z^{\prime}\be/Z}(X^{\prime}\e)$ endowed with a morphism of $Z^{\prime}$-schemes $q_{\le X^{\prime}}\colon \Re_{Z^{\prime}\be/Z}(X^{\prime}\e)\times_{Z}Z^{\prime} \to X^{\prime}$
such that the map
\begin{equation}\label{wr}
\Hom_{Z}\e(T,  \Re_{Z^{\prime}\be /Z}(X^{\prime}\e))\overset{\!\sim}{\to}\Hom_{Z^{\prime}}(\e T\!\times_{Z}\!Z^{\prime},X^{\prime} \e), \quad g\mapsto q_{\le X^{\prime}}\circ g_{\le Z^{\prime}} 
\end{equation}
is a bijection. This is the case if $f$ is finite and locally free and $X^{\prime}$ is quasi-projective over $Z^{\prime}$. See \cite[\S 7.6]{blr} and \cite[Appendix A.5]{cgp} for more details. Assume now that $f$ is a morphism of perfect $\mathbb F_{\be p}$-schemes and let $X^{\prime}$ be a perfect $Z^{\prime}$-scheme. Note that, if $T$ is a perfect $Z$-scheme, then $T\be\times_{\be Z}\be Z^{\prime}$ is a perfect $Z^{\prime}$-scheme by Remark \ref{top-prop}(d). We will say that the {\it the Weil restriction of $X^{\prime}$ along $f$  exists in $(\mr{Perf}/Z^{\prime}\e)$ or, more concisely, that $\Re_{\le Z^{\prime}\!/Z}^{\e\pf}(X^{\prime})$ exists,} if the contravariant functor
\begin{equation}\label{weil-pf}
(\mr{Perf}/Z)\to(\mathrm{Sets}),\quad T\mapsto\Hom_{\e\mr{Perf}/Z^{\prime}}\big(T\times_{Z}Z^{\prime},X^{\prime}\big),
\end{equation}
is represented by a perfect $Z$-scheme $\Re_{\le Z^{\prime}\!/Z}^{\e\pf}(X^{\prime})$.

\begin{lemma}\label{pf-wr} Let $f\colon Z^{\prime}\to Z$ be a  morphism of perfect $\mathbb F_{\be p}$-schemes and let $Y$ be a  $Z^{\prime}$-scheme. If $\Re_{\le Z^{\prime}\!/Z}(Y)\in(\mr{Sch}/Z\e)$ exists, then $\Re_{\le Z^{\prime}\!/Z}^{\pf}\big(Y^{\pf}\le\big)$ exists as well and
\[
\Re_{\le Z^{\prime}\!/Z}^{\pf}\big(Y^{\pf}\le\big)=\Re_{\le Z^{\prime}\!/Z}(Y)^{\pf}.
\]
\end{lemma}
\begin{proof} If $T$ is a perfect $Z$-scheme, then,  by  \eqref{pfc} and \eqref{wr},
\[
\begin{array}{rcl}
\Hom_{\e\mr{Perf}/Z^{\prime}}\big(T\times_{Z} Z^{\prime},Y^{\pf}\e\big)\!=\!\Hom_{\e\mr{Sch}/Z^{\prime}}(T\times_{Z} Z^{\prime},Y)&\!=\!&\Hom_{\e\mr{Sch}/Z}(T,\Re_{\le Z^{\prime}/Z}(Y))\\
&\!=\!&\Hom_{\e\mr{Perf}/Z}\big(T,\Re_{\le Z^{\prime}/Z}(Y)^{\pf}\e\big),
\end{array}
\]
whence the lemma follows. 
\end{proof}

\section{Exactness properties}\label{exa}
Let $k$ be a perfect field of positive characteristic and let $\kbar$ be an algebraic closure of $k$.

Let $(\le\mathrm{Perf}/k)_{\fpqc}$ be the category $(\le\mathrm{Perf}/k)$ endowed
with the fpqc topology, i.e., a family of morphisms $\{Y_{\alpha}\to Y\}$ in
$(\le\mathrm{Perf}/k)$ is a covering in $(\le\mathrm{Perf}/k)_{\fpqc}$ if it is a
covering in $(\mathrm{Sch}/k)_{\fpqc}$. The presheaf represented by a perfect and commutative $k$-group scheme is a sheaf on $(\mathrm{Sch}/k)_{\fpqc}$ \cite[Theorem 2.55, p.~34]{v} and therefore also on $(\le\mathrm{Perf}/k)_{\fpqc}$.

\begin{theorem}\label{ex-pf}
Let $1\to F\to G\to H\to 1$ be a sequence of $k$-group schemes which is exact for the fpqc topology on $(\mathrm{Sch}/k)$. Then the sequence of perfect $k$-group schemes
\[
1\to F^{\le\pf}\to G^{\le\pf}\to H^{\le\pf}\to 1
\]
is exact for the fpqc topology on $(\le\mathrm{Perf}/k)$.
\end{theorem}
\begin{proof}
Left-exactness follows from \eqref{pf}. Thus we are reduced to checking that $h_{\lbe S}(q^{\pf})$ is locally surjective in $(\mathrm{Perf}/k)_{\fpqc}^{\sh}$, where $S=\spec k$ and $q\colon G\to H$ is the given morphism.
Let $Z$ be a perfect $k$-scheme and let $f\in\Hom_{\le \mathrm{Perf}/k}\big(Z,H^{\pf}\le\big)$. By \eqref{pf}, $f=e^{\pf}$ for a unique $e\in \Hom_{\le \mathrm{Sch}/k}\big(Z,H\le\big)$. Now, since  $h_{\lbe S}(q^{\pf})$ is locally surjective in $(\mathrm{Sch}/k)_{\fpqc}^{\sh}$, there exist an fpqc covering $\{u_{i}\colon Y_{i}\to Z\}$ in $(\mathrm{Sch}/k)$ and $k$-morphisms $g_{\le i}\colon Y_{i}\to G$ such that $q\circ g_{\le i}=
e\circ u_{i}$ for every $i$. Since $\{u^{\pf}_{i}\colon Y^{\pf}_{i}\to Z\}$ is an fpqc covering in $(\le\mathrm{Perf}/k)$ by Remark \ref{top-prop}(c) and $q^{\pf}\be\circ\be g_{\le i}^{\le \pf}=e^{\pf}\circ \be u_{i}^{\pf}=f\be\circ\be u_{i}^{\pf}$ for every $i$, the proof is complete.
\end{proof}

\begin{corollary}\label{isg}
Let $f\colon G\to H$ be an isogeny of $k$-group schemes of finite type with geometrically connected kernel. Then $f^{\pf}\colon G^{\pf}\to H^{\pf}$ is an isomorphism.
\end{corollary}
\begin{proof} Let $F=\krn f$ and write $\bar{F}$ for the connected $\kbar$-scheme $F\times_{k}\spec \kbar$. By the theorem and Lemma \ref{pf=1}, it suffices to check that $F\lbe\big(\e\kbar\e\big)$ is the trivial group. Since $\bar{F}\to\spec\kbar$
a finite morphism of schemes, $|\bar{F}|$ is a finite connected topological space \cite[II, Corollary 6.1.7]{ega}, i.e., a one-point space. Since $F\lbe\big(\e\kbar\e\big)$ may be identified with a subset of $|\bar{F}\e|$ by \cite[(3.5.5), p.~243]{ega1}, the corollary is clear. 
\end{proof}

\begin{proposition}\label{ex-pf2}
Let $0\to F\stackrel{f}{\to} G\stackrel{g}{\to} H\to 0$ be a complex of commutative $k$-group schemes locally of finite type. Assume that the following conditions hold:
\begin{enumerate}
\item[(i)] $f$ is quasi-compact,
\item[(ii)] $g$ is flat,
\item[(iii)] the induced sequence of abelian groups $0\to F\lbe(\e\kbar\e)\to G\lbe(\e\kbar\e)\to H\lbe(\e\kbar\e)\to 0$ is exact.
\end{enumerate}
Then the induced complex of perfect and commutative $k$-group schemes
\[
0\to F^{\le\pf}\to G^{\le\pf}\to H^{\le\pf}\to 0
\]
is exact for the fpqc topology on $(\le\mathrm{Perf}/k)$.
\end{proposition}
\begin{proof} By faithfully flat and quasi-compact descent \cite[${\rm{IV}}_{2}$, Proposition 2.7.1]{ega}, it suffices to check the exactness of the indicated complex after extending the base from $k$ to $\kbar$. Then, using Remark \ref{top-prop}(d) and noting that conditions (i)-(iii) are stable under this base change, we may assume that $k=\kbar$. Since $g$ is flat and $g(k)\colon G(k)\to H(k)$ is surjective, the sequence $0\to(\krn g)^{\pf}\to G^{\pf}\to H^{\pf}\to 0$ is exact by Corollary \ref{ff} and Proposition \ref{ex-pf}. On the other hand, by (i) and \cite[${\rm VI_{A}}$, comments following Proposition 5.4.1]{sga3}, $F/\le\krn f$ and $\img f$ exist in the category of commutative $k$-group schemes and $f$ induces an isomorphism $F/\le\krn f\simeq\img f$, where $\img f$ is a closed subgroup scheme of $\krn g$. Thus, since $(\krn f)(k)=0$ by (iii), Proposition \ref{ex-pf} and Lemma \ref{pf=1} show that $f^{\pf}$ induces an isomorphism of perfect and commutative $k$-group schemes $F^{\e\pf}\simeq(\img f)^{\pf}$. It remains to show that $(\img f)^{\pf}=(\krn g)^{\pf}$. Since $\krn g\to \krn g/\e\img f$ is faithfully flat and locally of finite presentation \cite[${\rm VI_{A}}$, Theorem 3.3.2(ii)]{sga3}, $(\krn g)(k)\to(\krn g/\e\img f)(k)$ is surjective by \cite[$\text{IV}_{1}$, Proposition 1.3.7]{ega}. Thus Corollary \ref{ff} shows that the sequence $0\to \img f\to \krn g\to \krn g/\e\img f\to 0$ is exact for the fpqc topology on $(\le\mathrm{Sch}/k)$. Now Proposition \ref{ex-pf} reduces the proof to checking that
$(\krn g/\e\img f)^{\pf}=0$. Since $(\krn g)(k)=\krn g(k)=\img f(k)\subseteq (\img f)(k)$ by (iii), the map $(\img f)(k)\to(\krn g)(k)$ is surjective and therefore $(\krn g/\e\img f)(k)=0$. Now Lemma \ref{pf=1} completes the proof.
\end{proof}


\begin{thebibliography}{[28]}


\bibitem[BGA]{gfr} Bertapelle, A. and Gonz\'alez-Avil\'es, C.: The Greenberg functor revisited. Preliminary version available at arXiv:1311.0051v3.

\bibitem[BS]{bs} Bhatt, B. and Scholze, P.:
Projectivity of the Witt vector affine Grassmannian.  arXiv:1507.06490.

\bibitem[BLR]{blr} Bosch, S., L\"utkebohmert, W. and Raynaud, M.: N\'eron models. Ergeb. Math. Grenzgeb. {\bf{21}}, Springer-Verlag,  Berlin, 1990.

\bibitem[Bou]{bou} Bourbaki, N.: Commutative Algebra. Chapters 1--7. Softcover edition of the 2nd printing 1989, Springer-Verlag,  Berlin, 1998. ISBN 3-540-64239-0.

\bibitem[Bou2]{bou2} Bourbaki, N.: Algebra II. Chapters 4--7. Softcover printing of the first English edition of 1990, Springer-Verlag,  Berlin, 2003. ISBN 3-540-00706-7.

\bibitem[BD]{bd} Boyarchenko, M. and Drinfeld, V.: 
Character sheaves on unipotent groups in positive characteristic: foundations. 
Selecta Math. (N.S.) {\bf{20}}  (2014) 125--235.

\bibitem[BW]{bw} Boyarchenko, M. and Weinstein, J.: 
Maximal varieties and the local Langlands correspondence for ${\rm GL}(n)$.  
J. Amer. Math. Soc. {\bf{29}} (2016)  177--236. 

\bibitem[CGP]{cgp} Conrad, B., Gabber, O. and Prasad, G.:
Pseudo-reductive groups. New mathematical monographs {\bf{17}}, Cambridge Univ. Press, Cambridge, 2010.
 
\bibitem[DG]{dg} Demazure, M. and Gabriel, P.:  Groupes
alg\'ebriques. Tome I: G\'eom\'etrie alg\'ebrique,
g\'en\'eralit\'es, groupes commutatifs. Masson \& Cie, \'Editeur,
Paris, 1970. (with an Appendix by M. Hazewinkel: Corps de classes
local). ISBN 7204-2034-2.
 
\bibitem[$\text{SGA3}_{\le\text{new}}$]{sga3}  Demazure, M. and
Grothendieck, A. (Eds.): Sch\'emas en groupes. S\'eminaire de
G\'eom\'etrie Alg\'ebrique du Bois Marie 1962-64 (SGA 3). Augmented and
corrected 2008-2011 re-edition of the original by P. Gille and P. Polo.
Available at \url{http://www.math.jussieu.fr/~polo/SGA3}. Reviewed
at \url{http://www.jmilne.org/math/xnotes/SGA3r.pdf}.

 
\bibitem[Fon]{fon} Fontaine, J.-M.: Le corps de p\'eriodes $p$-adiques.
Ast\'erisque {\bf 223} (1994) 59--111.





\bibitem[Gre]{gre} Greenberg,  M. J.: Perfect closures of rings
and schemes. Proc. Amer. Math. Soc. {\bf{16}} (2) (1965) 313--317.
 
\bibitem[$\text{EGA I}_{\le\text{new}}$\e]{ega1} Grothendieck, A. and
Dieudonn\'e, J.: \'El\'ements de g\'eom\'etrie alg\'ebrique I. Le langage des
sch\'emas. Grundlehren Math. Wiss. {\bf{166}}, Springer-Verlag, Berlin, 1971.


\bibitem[EGA]{ega} Grothendieck, A. and Dieudonn\'e, J.: \'El\'ements de g\'eom\'etrie alg\'ebrique. Publ. Math. IHES {\bf{8}} ($=\text{EGA II}$) (1961), {\bf{11}} ($=\text{EGA IV}_{1}$) (1964), {\bf{24}} ($=\text{EGA IV}_{2}$) (1965), {\bf{28}} ($=\text{EGA IV}_{3}$) (1966), {\bf{32}} ($=\text{EGA IV}_{4}$) (1967).

\bibitem[$\text{SGA5}$]{sga5} Grothendieck, A. et al.: Cohomologie l-adique et fonctions $L$. S\'eminaire de G\'eom\'etrie Alg\'ebrique du Bois Marie 1965-66 (SGA 5). Lect. Notes in Math. {\bf{589}}, Springer-Verlag Berlin-New York, 1977.

\bibitem[Hart]{hart} Hartshorne, R.: 
Algebraic Geometry. Graduate Texts in Math.  {\bf{52}}, Springer-Verlag, New York, 1977.


\bibitem[KL]{kl} Kedlaya, K. S. and Liu, R.:
Relative $p$-adic Hodge theory: foundations. 
Ast\'erisque {\bf 371}, Paris: Soci\'et\'e Math\'ematique de France (2015).


\bibitem[Liu]{liu} Liu, Q.: 
Algebraic Geometry and Arithmetic Curves,  Oxford Graduate Texts in Mathematics {\bf{6}}, Oxford University Press, Oxford, 2002.

\bibitem[Lip]{lip} Lipman, J.:
The Picard group of a scheme over an Artin ring,  Publ. Math. IHES {\bf{46}} (1976), 15--86.

\bibitem[Mac]{mac}  Mac Lane, S.: Categories for the working mathematician.  Graduate Texts in Math. {\bf{5}}, Springer-Verlag, New York-Berlin, 1971.

\bibitem[Mat2]{mat2}  Matsumura, H.: Commutative Algebra. Second edition.
Math. Lect. Notes Series {\bf{56}}, The Benjamin/Cummings Publishing Company, Inc., 1980. ISBN 0-8053-7024-2.

\bibitem[Mi1]{mi1} Milne, J.S.: 
\'Etale Cohomology,  Princeton Mathematical Series {\bf{33}},  Princeton University Press, Princeton, N.J., 1980.

\bibitem[Mi2]{mi2} Milne, J. S.: 
Arithmetic Duality Theorems, Second Edition.   BookSurge, LLC, Charleston, SC, 2006. 

\bibitem[Pep]{pep} P\'epin, C.: Dualit\'e sur un corps local de caract\'eristique positive \`a corps r\'esiduel alg\'ebriquement clos.   arXiv:1411.0742v1.

\bibitem[Sch]{sch} Scholze, P.:
Perfectoid spaces,  Publ. Math. IHES {\bf{116}} (2012), 245--313.

\bibitem[Ser1]{ser1} Serre, J.-P.: Groupes proalg\'ebriques. Inst. Hautes \'Etudes Sci. Publ. Math. {\bf{7}} (1960), 5--67.

 \bibitem[Ser2]{ser2} Serre, J.-P.: Sur les corps locaux \`a corps r\'esiduel alg\'ebriquement clos. Bull. Soc. Math. France 89 (1961) 105--154.
 
\bibitem[Vis]{v} Vistoli, A.: 
Grothendieck topologies, fibered categories and descent theory, in: Fundamental Algebraic Geometry. Mathematical Surveys and Monographs, 123, American Mathematical Society, Providence, 2005, 1--104. 

\end{thebibliography}
\end{document}